\documentclass[a4paper,11pt]{article}
\usepackage{courier}
\usepackage{blindtext} 
\usepackage{amsmath}
\usepackage{amssymb}
\usepackage[colorlinks=true,breaklinks]{hyperref}
\usepackage[hyphenbreaks]{breakurl}
\usepackage{xcolor}
\definecolor{c1}{rgb}{0,0,1}
\definecolor{c2}{rgb}{0,0.3,0.9}
\definecolor{c3}{rgb}{0.3,0.9}
\hypersetup{linkcolor={c1},citecolor={c2},urlcolor={c3}}
\usepackage{enumerate}
\usepackage{todonotes}
\usepackage{makeidx}
\makeindex
\usepackage[top=3.2cm,bottom=3.2cm,left=2.5cm,right=2.5cm]{geometry}
\usepackage{fancyhdr}
\newcommand{\bbR}{\mathbb{R}}
\newcommand{\R}{\bbR}

\def\XXint#1#2#3{{\setbox0=\hbox{$#1{#2#3}{\int}$ }
\vcenter{\hbox{$#2#3$ }}\kern-.6\wd0}}

\usepackage{amsthm}
\theoremstyle{plain}
\newtheorem{theorem}{Theorem}[section]

\theoremstyle{definition}

\theoremstyle{lemma}
\newtheorem{lemma}[theorem]{Lemma}

\theoremstyle{Remark}
\newtheorem{Remark}[theorem]{Remark}

\theoremstyle{proposition}

\theoremstyle{corollary}
\newtheorem{corollary}[theorem]{Corollary}

\theoremstyle{example}

\theoremstyle{assumption}

\definecolor{brightred}{rgb}{1.0, 0.0, 0.0}
\usepackage{systeme}
\usepackage{colonequals}
\usepackage{comment}
\begin{document}
\pagestyle{empty}

\title{Long time behavior of Fokker-Planck equations for bosons and fermions }

\author{Anton Arnold\thanks{Institut f\"ur Analysis und Scientific Computing, Technische Universit\"at Wien, Wiedner Hauptstr. 8-10, A-1040 Wien, Austria, {\tt anton.arnold@tuwien.ac.at.}}, \,\, Marlies Pirner\thanks{Institut f\"ur Analysis und Numerik, Fachbereich Mathematik
und Informatik der Universit\"at M\"unster, Orl\'eans-Ring 10, 48149 M\"unster, Germany, {\tt marlies.pirner@uni-muenster.de.}}, \,\, Gayrat Toshpulatov\thanks{Institut f\"ur Analysis und Numerik, Fachbereich Mathematik
und Informatik der Universit\"at M\"unster, Orl\'eans-Ring 10, 48149 M\"unster, Germany, {\tt  gayrat.toshpulatov@uni-muenster.de}.}}

\maketitle
%%%%%%%%%%%%%%%%%%%%%%%%%%%%%%%%%%%%%%%%%%%%%%%%%%%%
\pagestyle{plain}

\begin{abstract} 
This paper is concerned with space inhomogeneous quantum Fokker-Planck equations posed on a classical kinetic phase space. The nonlinear factor $f(1\pm f)$ appears both in the transport term and in the collison part of the Fokker-Planck operator, accounting for the inclusion principle of bosons and the exclusion principle of fermions. 

Assuming that global solutions exist, we prove exponential decay of the solutions to the global equilibrium in a weighted $L^2$-space without a close-to-equilibrium assumption. Our analysis is in the spirit of an $L^2$-hypocoercivity method. Our main Lyapunov functional is constructed from a logarithmic relative entropy and the (nonlinear) projection of the solution to the manifold of local-in-$x$ equilibria. 

\end{abstract}
%%%%%%%%%%%%%%%%%%%%%%
\textcolor{black}{\begin{small}\textbf{Keywords:} Semi-classical kinetic equations, Fokker-Planck equation, fermion, boson, long time
behavior, convergence to equilibrium, hypocoercivity, Lyapunov functional.\\
\textbf{2020 Mathematics Subject Classification:} 35Q84, 35B40, 35Q40, 35Q82, 82C40.
\end{small}}
%%%%%%%%%%%%%%%%%%%%%%%%%%%%%
%there should be an intruduction
%%%%%%%%%%%%%%%%%%%%%%%%%%%%%%%%%%%%%%%%%%%%%%%%%%%
%\newpage
\tableofcontents
\section{Introduction}
One way to study the statistical evolution of rarefied gases of particles consists in writing down kinetic equations for the distribution function $f$ of the particles in phase space. These equations consist of transport and collision parts.  With the invention of quantum mechanics, kinetic  equations have been modified in order to describe also quantum statistical effects like  Pauli's exclusion-inclusion principle. This approach was initiated by the physicists Nordheim \cite{Nordh}, Uehling and
Uhlenbeck \cite{U.U}.  Several kinetic equations for interacting particles with Pauli's exclusion-inclusion principle, such as fermions and bosons, have been introduced in the literature, see  e.g.\,\cite{Book, Book1, Dolb, E.M1, E.M.Vel, Lu.Wenn, Vhb}. Such models are now often referred as \emph{quantum Boltzmann equations} and they have been used in applications like the electron theory of conductivity \cite{Nordh} and electron interactions in semiconductor devices \cite{Book1}.
 A notable feature of quantum kinetic equations is the appearance of  the term $1\pm f$ in their collision part, which eventually gives rise to the quantum mechanical Bose-Einstein and, respectively, Fermi-Dirac statistics of their equilibrium distribution.
It was shown in \cite{Contr} that adding the term $1\pm f$ only to the collision part of the kinetic equations can be insufficient to describe accurately quantum systems far from equilibrium. To deal with this problem, Balazs \cite{Der} and  
Kaniadakis \cite{K} (see also \cite{KQ1, KQ2}) introduced the following kinetic model which includes the term $1\pm f$ not only in the collision part of Fokker-Planck type but also in the transport part 
\begin{equation}\label{Eq}
\begin{cases}
\partial_t f+p\cdot \nabla_x (f+\kappa f^2)=\text{div}_{p}(\nabla_pf+p(f+\kappa f^2))\\
f(0,x,p)=f_0(x,p).
\end{cases}
\end{equation}
The term $1+\kappa f$  appearing in \eqref{Eq} takes into account the exclusion principle for fermions when $\kappa=-1$ and the inclusion principle for bosons when $\kappa=1.$ Moreover, the values of the distribution function have to respect  the bounds $
0\leq f\leq 1$
when $\kappa=-1.$ Clearly, for $\kappa=0,$ \eqref{Eq} reduces to the classical, kinetic Fokker-Planck equation.

In this paper, we consider this  nonlinear, spatially inhomogeneous kinetic equation. Here the variable $t\geq 0,$ $x \in \mathbb{T}^d,$ and $p \in \mathbb{R}^d,$ respectively, stand for time, spatial position, and momentum. The unknown $f(t,x,p)\geq 0$ describes the time evolution of  the distribution function of indistinguishable quantum particles, which are fermions or bosons. 
 
 While \eqref{Eq} may be called a "quantum Fokker-Planck equation", we note that several other models from the literature are also referred to by this very name. %Many of them 
 Fully quantum mechanical descriptions of the transport are either evolution equations for the quantum kinetic Wigner function \cite{Wig1, JMH} or evolution equations in Lindblad form for the density matrix \cite{Wig2, Wig3}. While a Wigner-Fokker-Planck equation (without external potential) would be very similar to \eqref{Eq} with the choice $\kappa=0$, they have intrinsically different settings: A Wigner equation is posed on a quantum mechanical phase space and, when describing an open quantum system, it (or its density matrix analog) has to be in Lindblad form \cite{Wig1}. By contrast, the Fokker-Planck equations \eqref{Eq} are posed in a classical phase space, and they do not correspond to the Lindblad form. Although equation \eqref{Eq} is posed on a classical phase space, it includes at least quantum statistical effects through the term $1+\kappa f.$ In this sense, equation \eqref{Eq} can be understood as a semi-classical kinetic model: the evolution of the particles is described by a distribution function on a classical phase space, while the exclusion-inclusion principle is included through the term $1+\kappa f.$ Such models are used to incorporate quantum statistical effects into a kinetic description, e.g.\ to include Fermi exclusion effects in a plasma \cite{KQ1}. Alternatively, the bosonic Fokker–Planck equation has been used to describe the evolution of bosons and to study the formation of a Bose–Einstein condensate \cite{Num,To}. For a more detailed discussion of the physical interpretation and applications of this model, we refer to \cite{KQ1,KQ2}.

For both values of $\kappa,$ Equation \eqref{Eq} satisfies the  following  physical principles. First,
whenever $f(t,x,p)$ is a (well-behaved) solution of \eqref{Eq}, one has  \textit{conservation of mass:}
\begin{equation*}
\int_{\mathbb{T}^d}\int_{\mathbb{R}^d}f(t,x,p)dpdx=\int_{\mathbb{T}^d}\int_{\mathbb{R}^d}f_0(x,p)dpdx, \, \, \, \, \forall\, t\geq 0.
\end{equation*}
Next, Equation \eqref{Eq}  is \emph{dissipative} in the sense that the following entropy functional  decreases under time-evolution of $f:$ we define  
\begin{equation*}\mathrm{H}[f]\colonequals  \int_{\mathbb{T}^d}\int_{\mathbb{R}^d} \frac{|p|^2}{2}f dpdx+\int_{\mathbb{T}^d}\int_{\mathbb{R}^d} \left(f \ln{f}-\kappa(1+\kappa f)\ln{(1+\kappa f)}\right)dpdx,
\end{equation*}
which is bounded below inspite of the non-positive second summand (see \cite[Lemma 3.1]{FDFP} for $\kappa=-1$ and \cite[Lemma 4.5]{2D}  for $\kappa=1$). 
If  $f$  is a regular enough solution, one obtains for all $t>0,$\begin{align}\label{dt H2}
\frac{d}{dt}\mathrm{H}[f(t)]=-\int_{\mathbb{T}^d}\int_{\mathbb{R}^d}(f+\kappa f^2)\left|\nabla_p \ln\left(\frac{f}{M(1+\kappa f)}\right)\right|^2dpdx\leq 0,
\end{align}
where $M(p)\colonequals \frac{1}{(2\pi )^{d/2}}e^{-\frac{|p|^2}{2}}.$
 This means that the functional  $\mathrm{H}[f(t)]$ is a decreasing function of $t\geq 0,$ hence it is a Lyapunov functional for \eqref{Eq}. The entropy dissipation functional  (the integral on the right hand side of \eqref{dt H2})  vanishes if and only if $\frac{f}{M(1+\kappa f)}$ does not depend on $p.$ This means 
\begin{equation}\label{loc}
f(t,x,p)=\frac{\beta(t,x) M(p)}{1-\kappa \beta(t,x)M(p)}
\end{equation}
for some function $\beta=\beta(t,x).$ Hence, any function  $f$ in the form of \eqref{loc} is a \emph{local equilibrium} for \eqref{Eq}. For such $f, $ the left hand side of \eqref{Eq} vanishes if $\beta$ is constant. Hence,   the unique  \emph{ global equilibrium} is 
\begin{equation}\label{eqilib}
f_{\infty}(p)\colonequals \frac{\beta_{\infty}M(p)}{1-\kappa \beta_{\infty}M(p)},
\end{equation}
 where $\beta_{\infty}>0$ is uniquely determined by mass conservation 
 \begin{equation}\label{eqilib-mass}
 \int_{\mathbb{T}^d}\int_{\mathbb{R}^d}f_{0}(x,p)dpdx=\displaystyle |\mathbb{T}^d|\int_{\mathbb{R}^d}f_{\infty}(p)dp.
\end{equation}
The function $f_{\infty}(p)$ is called the \emph{Fermi-Dirac} distribution when $\kappa=-1$ and the \emph{Bose-Einstein} distribution when $\kappa=1.$
We note that $f_{\infty}$ is not always well-defined when $\kappa=1$ and $d\geq 3:$ since $f_{\infty}$ is supposed to be non-negative and have finite values, $\beta_{\infty}$ needs to be in the interval $(0, (2\pi )^{d/2}),$ see \cite{2D}. This means $0\leq f_{\infty}(p)\leq \frac{e^{-|p|^2/2}}{1-e^{-|p|^2/2}}$ for all $p \in \mathbb{R}^d,$ hence  the mass of $f_{\infty}$ can not be larger than
\begin{equation}\label{m_c}m_c\colonequals \int_{\mathbb{T}^d}\int_{\mathbb{R}^d}\frac{e^{-|p|^2/2}}{1-e^{-|p|^2/2}}dpdx=|\mathbb{T}^d|\int_{\mathbb{R}^d}\frac{e^{-|p|^2/2}}{1-e^{-|p|^2/2}}dp.
\end{equation}
One can check that   $m_c=\infty$ if $d=1,2$ and $m_c<\infty$ if $d\geq 3.$ Therefore, there is exactly one global equilibrium for each fixed finite mass in dimension $d=1,2.$ However, in dimension $d\geq 3,$   there is no (smooth) global equilibrium if the mass of $f_0$ is larger than $m_c.$

On the basis of the decay of $\mathrm{H}[f(t)],$ one can conjecture  that $f(t)$ converges to the global equilibrium  $f_{\infty}$ as $t\to \infty.$   Thus, it is an interesting problem to prove (or disprove) this convergence and to estimate the convergence rate. Our goal is to study the long time behavior of the  solution $f(t),$  and to prove that $f(t)$ converges exponentially to $f_{\infty}.$

If we simplify \eqref{Eq} by neglecting the dependence of the solution on the $x$ variable, then we get the spatially homogeneous equation
\begin{equation}
\label{Eqhom}
\begin{cases}
\partial_t f(t,p)=\text{div}_{p}(\nabla_pf(t,p)+p(f(t,p)+\kappa f^2(t,p)))\\
f(0,p)=f_0(p).
\end{cases}
\end{equation} 
 This equation has been studied comprehensively: exponential decay of the solution to the equilibrium in dimension $d=1$  was proven in \cite{1D} using the "entropy-entropy dissipation" method \cite{EED, EEDn}. The analogous convergence result and well-posedness in higher dimensions  were obtained in \cite{FDFP} for fermions and $d\in \mathbb{N},$ and in \cite{2D} for bosons and $d=2.$  As we mentioned before, for the bosonic case, i.e. $\kappa=1,$ there is no  global equilibrium if $f_0$ has mass larger than the critical mass $m_c$ (see \eqref{m_c}) in dimension $d\geq 3.$ In this case, it is conjectured
that solutions  will blow up in   finite time. This  problem was studied only for the  spatially homogeneous equation \eqref{Eqhom}:  Toscani \cite{To} proved that in dimension
$d\geq 3$ there are certain solutions which blow up if the initial mass is sufficiently large or, for any given
supercritical mass, if the initial second moment is small enough. The continuation beyond this blow-up and long time behavior was studied in \cite{Num, Hopf}. 
 
 Concerning the spatially inhomogeneous equation \eqref{Eq}, there are only few studies: existence of smooth solutions and their exponential convergence  to the global equilibrium was proven in \cite{NS}  when the initial data  is close to equilibrium  in a suitable Sobolev space. This 
result is based on the $H^1$ hypocoercivity method developed in previous works \cite{NM, V}. Luo and Zhang \cite{Chi} considered \eqref{Eq}  in whole space, i.e. $x \in \mathbb{R}^3. $ Using the high-order Sobolev energy method developed by Guo \cite{Guo}, they obtained  algebraic decay rates   under the assumption that the initial data is close to equilibrium  in a Sobolev space.

In this paper, we shall improve these previous results. We prove exponential decay of the solution  to the global equilibrium without the close-to-equilibrium assumption.
For the spatially homogeneous equation \eqref{Eqhom} it is known that explicit estimates on convergence to equilibrium can be obtained by a direct study of the entropy and its  dissipation functional. Since $f=f(t,p)$ does not depend on $x$ in this case, with the help of convex Sobolev inequalities \cite{EEDn} one can obtain a Gr\"onwall inequality for the relative entropy. This implies exponential convergence of the solution to the global equilibrium, see \cite{1D, FDFP, 2D}. Unfortunately,  this idea cannot be used for the spatially inhomogeneous  equation \eqref{Eq}. The reason is that one loses information about the $x-$direction in the dissipation functional. More precisely, there is no gradient with respect to $x$ in the dissipation functional (see \eqref{dt H2}), hence a Sobolev inequality does not hold for the entropy and its dissipation functional.

 To deal with this problem, in  recent years, many new so-called hypocoercivity methods have been introduced to study the long-time behavior of spatially inhomogeneous kinetic equations.
The challenge of hypocoercivity is to understand the interplay between the collision operator that provides dissipativity in the velocity variable and the transport operator which is conservative, in order to obtain global dissipativity for the whole problem. 
There are several hypocoercivity methods, see e.g. \cite{NM, V, Guo, DMS, An, AT1, AT2,  Har}. In this work we use the $L^2$-hypocoercivity method which was
introduced in \cite{Herau, DMS} for linear kinetic equations which only conserve mass. This method was developed further in \cite{GuoL2, Max, L2} for linear kinetic equations with several conserved quantities (mass, momentum, energy).

The main motivation of this work is to provide an extension of the hypocoercivity results  from \cite{NS, Chi} for the model \eqref{Eq}: We shall prove  exponential decay of the solution  to the global equilibrium in $L^2$ without the close-to-equilibrium assumption on the initial data. 
Motivated by the recent results \cite{FPS, PT, G.T}, we construct a Lyapunov functional for \eqref{Eq} by  modifying the  $L^2$-hypocoercivity method.
While these three papers establish various techniques for handling nonlinear collision or diffusion operators in kinetic equations, their equations all include the standard linear transport operator. By contrast, \eqref{Eq} includes the nonlinear term $f(1+\kappa f)$ in both the transport and diffusion term, which poses the key challenge for this work, requiring new modifications in the construction of a Lyapunov functional. A further difficulty is related to the structure of the local equilibria. In the classical setting, the projection onto the manifold of local equilibria is linear, which plays an important role in the $L^2$-hypocoercive method.  For  \eqref{Eq}, this projection is nonlinear and in the fermionic case similar to the projection in \cite{PT}; see \eqref{pr} below. Consequently, the macroscopic-microscopic decomposition itself depends nonlinearly on the solution, making it considerably more difficult to control the interaction between the dissipative and transport components.
These features prevent a direct application of existing hypocoercivity techniques and require new ideas for the construction and analysis of a suitable Lyapunov functional. In this work, we develop such a method by modifying the $L^2$-hypocoercivity method and combining it with the relative entropy and the nonlinear projection onto the manifold of local equilibria. The resulting Lyapunov functional is equivalent to the squared $L^2$-distance from the global equilibrium and satisfies a Gr\"onwall-type inequality. This allows us to establish exponential convergence to equilibrium for general admissible initial data, without requiring the initial data to be a small perturbation of the global equilibrium.

  %In general, for semi-classical kinetic equations this projection operator  is nonlinear (see \eqref{pr} below).  This makes our analysis challenging. 
%\textcolor{blue}{And here maybe one or two sentences on the difficulties and the techniques. I think we should cite \cite{FPS} and \cite{BHR} in this context.}

The organization of this paper is as follows. In Section 2 we state the main result and explain the main steps of the proof. In Section 3 we study some properties of weak solutions. We show that such solutions conserve mass and their $L^1-$difference is contracting as time progresses. In Section 4 we establish some estimates   for the projection operator and the  entropy dissipation, construct a Lyapunov functional, and   prove the main result. 
%%%%%%%%%%%%%%%%%%%%%%%%%%%%%%%%%%%%%%%%%%%%%%%%%%%%%%%%%%%%%%%%%%%%%%%%%%%%%%%%%%%%%%%%%%%%%%%%%%%%%%
\section{Main result}
In this paper we consider the long time behavior of global weak solutions.
We define a global weak solution
\begin{equation}\label{f}
f\in C([0, \infty), L^1 \cap L^{2}(\mathbb{T}^d\times \mathbb{R}^d))
\end{equation} of \eqref{Eq}, when  satisfying\footnote{We remark that this notion of \emph{weak solution} in \eqref{dist} is the continuous analog of an \emph{ultra-weak variational formulation} as used for numerical discretizations of elliptic problems, see \cite{Cessenat-Despret}.} \begin{equation}\label{dist}
\int_{0}^{\infty}\int_{\mathbb{T}^d}\int_{\mathbb{R}^d}f\big[\partial_t \varphi+(1+\kappa f)p\cdot \nabla_x \varphi-(1+\kappa f)p\cdot \nabla_p \varphi+\Delta_p \varphi\big]dpdxdt+\int_{\mathbb{T}^d}\int_{\mathbb{R}^d}\varphi_{|t=0}f_0dpdx=0
\end{equation}
  for all $\varphi\in C^{\infty}_{c}([0,\infty)\times \mathbb{T}^d\times \mathbb{R}^d).$  It is natural to require the condition  \eqref{f} on $f,$ because it implies that  all  terms of \eqref{dist} are finite for all $\varphi\in C^{\infty}_{c}([0,\infty)\times \mathbb{T}^d\times \mathbb{R}^d).$ 
 
 As the existence of smooth or weak solutions to \eqref{Eq} has not been established yet, our subsequent analysis is under the assumption that they exist globally in time.  But, at least, we can apply our results to  global classical solutions near
the equilibrium in a torus  constructed by Neumann, Sparber \cite{NS} and Luo, Zhang \cite{Chi}. 
 
We introduce the weighted $L^2-$space 
\begin{equation*}
L^2(\mathbb{T}^d\times \mathbb{R}^d, M^{-1})\colonequals \left\{g :\mathbb{T}^d\times \mathbb{R}^{d} \to \mathbb{R}: g \text{ is measurable and } \, \, \int_{\mathbb{T}^d}\int_{\mathbb{R}^d}\frac{g^2}{M}dpdx< \infty\right\}
\end{equation*}
with the scalar product
 \begin{equation*}
\langle g_1, g_2\rangle\colonequals \int_{\mathbb{T}^d}\int_{\mathbb{R}^d}\frac{g_1g_2}{M}dpdx, \, \, \, \, \, \, g_1,g_2 \in L^2(\mathbb{T}^d\times \mathbb{R}^d, M^{-1}) 
\end{equation*}
and the induced norm $||\cdot||.$

 Our main result is the following:
\begin{theorem}\label{main th} Let $d\in \mathbb{N}$. 
 Assume there exist constants $\beta_->0$ and $\, \beta_+>0$ such that $f_0$ satisfies
\begin{equation}\label{f_0}
\frac{\beta_{-} M(p)}{1-\kappa \beta_{-} M(p)}\leq f_0(x,p)\leq \frac{\beta_{+} M(p)}{1-\kappa \beta_{+} M(p)} 
\end{equation}
for almost every $(x, p)\in\mathbb{T}^d\times \mathbb{R}^{d} $.  We also require $\beta_{+}<(2\pi )^{d/2}$ when $\kappa=1.$ Let $f\in C([0,\infty), {L^1\cap L^{2}}(\mathbb{T}^d\times \mathbb{R}^d))$ be a weak solution to \eqref{Eq} such that
\begin{equation}\label{reg}
\nabla_x f, \nabla_p f  \in L^2_{loc}([0,\infty),  L^1\cap L^2(\mathbb{T}^d\times \mathbb{R}^d)).
\end{equation}
 Then the solution {$f$ satisfies the same bounds, i.e.
 \begin{equation}\label{-f+}\frac{\beta_{-} M(p)}{1-\kappa \beta_{-} M(p)}\leq f(t, x,p)\leq \frac{\beta_{+} M(p)}{1-\kappa \beta_{+} M(p)} 
\end{equation}
for almost every $(x, p)\in\mathbb{T}^d\times \mathbb{R}^{d} $ and  for all $t\geq 0.$  Moreover,}   there exist  explicitly computable constants $\lambda>0$ and $ c\geq 1$ (independent of $f_0$ but depending on $\beta_{-}$ and $ \beta_{+}$)  such that
\begin{equation}\label{lamda}
||f(t)-f_{\infty}||\leq c e^{-\lambda t}||f_0-f_{\infty}||, \, \, \, \, \forall\, t\geq 0,
\end{equation}
{where $f_\infty$ is determined by mass conservation along the flow, see \eqref{eqilib}-\eqref{eqilib-mass}}.
\end{theorem}
\begin{Remark} { $(i)$ 
  When $\kappa=-1,$ it is obvious that the functions on the left and right hand sides of  \eqref{f_0}  are smooth, bounded, and integrable. It is easy to check $\beta_{-}\leq \beta_{+}.$ {When $\kappa=1$ and under the assumption  $\beta_{+}<(2\pi )^{d/2}$, these functions have the same properties, as they have no singularity at $p=0$.
  Similarly,} $f_{\infty}(p)=\frac{\beta_{\infty}M(p)}{1-\kappa \beta_{\infty}M(p)} $ is also smooth, bounded, and integrable when $\kappa=-1$. The condition  $\int_{\mathbb{T}^d}\int_{\mathbb{R}^d}f_{\infty}dxdp=\int_{\mathbb{T}^d}\int_{\mathbb{R}^d}f_0dxdp$  and \eqref{f_0} imply $\beta_-\leq \beta_{\infty}\leq\beta_+$. Hence, {for $\kappa=1$,} the condition  $\beta_+<(2\pi)^{d/2}$  provides the same properties {of $f_\infty$ and $\beta_\infty$}. 
 \\
 $(ii)$ The bounds in \eqref{-f+}  show that $f\in L^{\infty}([0,\infty)\times \mathbb{T}^d\times \mathbb{R}^d)$ and $f(t)\in L^2(\mathbb{T}^d\times \mathbb{R}^d, M^{-1})$ for all $t\geq 0$. Moreover, when $\kappa=-1,$ we have $0<f(t,x,v)<1$ for almost every $(x, p)\in\mathbb{T}^d\times \mathbb{R}^{d} $ and  for all $t\geq 0.$ }
 \end{Remark}

\bigskip

The rest of the paper is devoted to the proof of Theorem \ref{main th}, structured in several  steps. Although we shall consider weak solutions, for illustration purposes,  we first explain these steps here for smooth solutions.
The first important step consists in proving the contraction of the $L^1-$difference of any two solutions, i.e., if $f$ and $g$ are two solutions, then
$$||f(t)-g(t)||_{L^1}\leq ||f_0-g_0||_{L^1},\,\,\,\, \forall\, t\geq 0$$
 (see Lemma \ref{max.pri}). Under the assumptions of Theorem \ref{main th}, this lets us prove that the solution $f$ satisfies  \begin{equation}\label{*1}\frac{\beta_{-} M(p)}{1-\kappa \beta_{-} M(p)}\leq f(t,x,p)\leq \frac{\beta_{+} M(p)}{1-\kappa \beta_{+} M(p)}
  \end{equation} for all $x \in \mathbb{T}^d,$  $p\in \mathbb{R}^d,$ and $t\geq 0$ (see Corollary \ref{Cor:bound}).  This bound plays an essential role in our estimates, in particular, this bound makes the relative entropy $$\mathrm{H}[f|f_{\infty}]\colonequals  \int_{\mathbb{T}^d}\int_{\mathbb{R}^d} \left(f \ln{\frac{f}{f_{\infty}}}-\kappa(1+\kappa f)\ln{\frac{1+\kappa f}{1+\kappa f_{\infty}}}\right)dpdx$$ equivalent to  $||f-f_{\infty}||^2$ (see Lemma \ref{C<H<C}). Next, we construct a Lyapunov functional $\mathrm{E}[f|f_{\infty}]$ by modifying the relative entropy, 
\begin{equation}\label{some}
\mathrm{E}[f|f_{\infty}]=\mathrm{H}[f|f_{\infty}]+\text{some terms}.
\end{equation}
We want this functional still to be equivalent to $||f-f_{\infty}||^2$ and to satisfy 
\begin{equation}\label{dt E C}
\frac{d}{dt} \mathrm{E}[f|f_{\infty}]\leq -2C ||f-f_{\infty}||^2,\,\,\,\forall\, t>0.
\end{equation}
Here  and in the following estimates, $C$ denotes  (possible different) positive constants.
With 
the equivalence of $\mathrm{E}[f|f_{\infty}]$ to  $||f-f_{\infty}||^2$ and \eqref{dt E C} we shall obtain
\begin{equation}\label{inE}\frac{d}{dt} \mathrm{E}[f|f_{\infty}]\leq -2C  \mathrm{E}[f|f_{\infty}],\,\,\,\forall\, t>0.
\end{equation}
The Gr\"onwall inequality then provides exponential decay of $\mathrm{E}[f|f_{\infty}]$. The  equivalence of $\mathrm{E}[f|f_{\infty}]$ to $||f-f_{\infty}||^2$ yields \eqref{lamda}.

The main challenge is the construction of such functional $\mathrm{E}[f|f_{\infty}]$, and we use the following idea: Using a generalized log-Sobolev inequality and the bounds \eqref{*1} on $f,$ we can estimate 
\begin{equation}\label{dtHint}
\frac{d}{dt} \mathrm{H}[f|f_{\infty}]\leq -C||f-\Pi [f]||^2,\,\,\,\forall\, t>0,
\end{equation}
where $\Pi [f]$ is the nonlinear projection of $f$ to the manifold of local equilibria. It is defined by \begin{equation*}
{\Pi [f]}(t,x,p)\colonequals\frac{\beta(t,x)M(p)}{1-\kappa \beta(t,x)M(p)},
\end{equation*}
where $\beta(t,x)$ is uniquely chosen by the condition $\displaystyle 
\int_{\mathbb{R}^d}\Pi [f]dp=\int_{\mathbb{R}^d} fdp,$ {i.e.\ the conservation of the local-in-$x$ mass density.} {The right hand side of \eqref{dtHint} vanishes for any function $f$ satisfying $f=\Pi [f]$ while $ \mathrm{H}[f|f_{\infty}]$  vanishes if and only if $f=f_{\infty}$. {Therefore,} the estimate   \eqref{dtHint}  {is not enough to write a closed differential inequality  for $ \mathrm{H}[f|f_{\infty}]$ ---} as in \eqref{inE} (or another Gr\"onwall-type differential inequality).} Therefore, \textit{"some terms"} in \eqref{some} has to be {non-zero} and chosen so that its time derivative provides $\displaystyle -C||\Pi [f]-f_{\infty}||^2.$   The reason is that the sum of this term and  the right hand side of  \eqref{dtHint}  can be estimated by $-C||f-f_{\infty}||^2,$ i.e., the triangle inequality provides  
$$-||f-\Pi [f]||^2-||\Pi [f]-f_{\infty}||^2\leq -\frac{1}{2}||f-f_{\infty}||^2.$$
We observe that the term $\displaystyle -C||\Pi [f]-f_{\infty}||^2$  can be obtained from the time derivative of the macroscopic flux $\displaystyle j\colonequals\int_{\mathbb{R}^d}pf dp.$ More precisely, we have
$$\partial_t j=-\nabla_x \int_{\mathbb{R}^d}(\Pi [f]-f_{\infty})dp+\text{some terms},$$
see Step 2 in the proof of Theorem 2.1, noting that $\int_{\mathbb{R}^d}(\Pi [f]-f_{\infty})dp=\rho-\rho_{\infty}$ (with $\rho\colonequals \int_{\mathbb{R}^d} fdp$). If $\phi$ denotes the solution of  $\displaystyle -\Delta_x \phi=\int_{\mathbb{R}^d}(\Pi [f]-f_{\infty})dp,$ then integration by parts shows
$$\int_{\mathbb{T}^d}\nabla_x \phi \cdot \partial_t jdx= - \int_{\mathbb{T}^d}\left(\int_{\mathbb{R}^d}(\Pi [f]-f_{\infty})dp\right)^2 dx+\text{some terms}. $$ 
 We shall actually show that  the term $-\displaystyle \int_{\mathbb{T}^d}\left(\int_{\mathbb{R}^d}(\Pi [f]-f_{\infty})dp\right)^2 dx=-||\rho-\rho_{\infty}||^2_{L^2(\mathbb{T}^d)}$   is enough to  control the term  $-C||\Pi [f]-f_{\infty}||^2.$
 Therefore, we construct the Lypunov functional $\mathrm{E}[f|f_{\infty}]$ depending on $j$ and $\nabla_x \phi$ as
\begin{equation*}
\mathrm{E}[f|f_{\infty}]\colonequals \mathrm{H}[f|f_{\infty}]+\delta \int_{\mathbb{T}^d}\nabla_x \phi \cdot  jdx, \,\,\, \, \, \text{with some }\,\,\,\,\delta>0. 
\end{equation*}
We show that this functional is equivalent to $||f-f_{\infty}||^2$ and satisfies \eqref{dt E C}. 
This leads to our main result \eqref{lamda}.
%%%%%%%%%%%%%%%%%%%%%%%%%%%%%%%%%%%%%%%%%%%%%%%%%%%
\section{Properties of weak solutions}
In this section we derive several key properties of weak solutions that will be needed for the decay analysis in Section 4.
%%%%%%%%%%%%%%%%%%%%%%%%%%%%%%%%%%%%%%%%%%%%%%%%%%%%%%%%%%%%%%%%%%%%%%%%%%
\subsection{Approximation of weak solutions by smooth functions}
%%%%%%%%%%%%%%%%%%%%%%%%%%%%%%%%%%%%%
 
We consider the function $r\in C^{\infty}_c(\mathbb{R}^d)$ given by
$$r(p)\colonequals \begin{cases}
e^{\frac{1}{|p|^2-1}},\, \, \,  &\text{if}\, \, \, \, |p|< 1\\
0, &\text{if}\,\, \, \, |p|>1.
\end{cases}$$
We define a sequence of smooth functions $r_n\in C_c^{\infty}(\mathbb{R}^d), \,  n\in \mathbb{N},$ given by
$$r_n(p)\colonequals \frac{1}{\int_{\mathbb{R}^d}r(p')dp'}n^dr(np),$$ 
{and }  
supported in $\{p\in \mathbb{R}^d:\, |p|<1/n\} $ 
Let $\tilde{r}_n\in C^{\infty}(\mathbb{T}^d)$ be a periodic function such that $\tilde{r}_n(x)=r_n(x)$ for $x\in [-1/2,1/2)^d.$  Note that $\tilde{r}_n(x)$ with $n>2$ is non-zero in $\{ x\in [-1/2,1/2)^d:\, |x|<1/n\} $  and zero in $\{x\in [-1/2,1/2)^d:\, |x|>1/n\}. $ 
Let $\tau_n \in C^{\infty}_c(\mathbb{R})$  have its support in the interval $(-1/n, 0)$ and tend to the Dirac function as $n\to \infty.$ We also assume $\tau_n\geq 0$ and $ \int_{\mathbb{R}}\tau_n(t)dt=1,$ $n\in \mathbb{N}.$ Then  $\eta_n(t,x,p)\colonequals \tau_n(t)\tilde{r}_n(x)r_n(p)$ forms a sequence of smooth mollifiers (see \cite[Section 4.4]{Brezis})\footnote{Actually, $\tilde{r}_n$ is smooth only for $n\geq 2,$ but this won't be a problem in the sequel.}.

Let $f\in C([0,\infty), L^1\cap L^{2}(\mathbb{T}^d\times \mathbb{R}^d))$ be a weak solution to \eqref{Eq}.  We define
\begin{equation}\label{f_n}
f_n\colonequals \eta_n*f,
\end{equation}
where  $f(t,x, p)$ is extended to $t\in (-\infty,0)$ by making it equal to zero. 
Young's inequality (see \cite[Proposition 8.9 and 8.10]{Foll}) provides  $$f_n\in C^{\infty}([0,\infty),W^{s,k} (\mathbb{T}^d\times \mathbb{R}^d)), \, \, \forall\, s\geq 0,\,\, \forall\, k\in[1,\infty].$$
 Replacing $(t,x,p)$ by $(t',x',p')$ in \eqref{dist} and then choosing $\varphi(t',x',p')=\eta_n(t-t', x-x',p-p')$ in \eqref{dist}, we get 
\begin{multline}\label{eq.fn0}
\partial_t f_n +p\cdot \nabla_x[\eta_n*(f+\kappa f^2)]-\text{div}_p[p (\eta_n*(f+\kappa f^2))]-\Delta_p f_n\\=\text{div}_x[(p\eta_n)*(f+\kappa f^2)]-\text{div}_p[(p\eta_n)*(f+\kappa f^2)].
\end{multline}
We write this equation as
\begin{equation}\label{eq.fn}
\partial_t f_n +p\cdot \nabla_x(f_n+\kappa f_n^2)-\text{div}_p[p (f_n+\kappa f_n^2)]-\Delta_p f_n=R_n(f),
\end{equation}
where 
\begin{equation}\label{R_n}R_n(f)\colonequals \text{div}_x[(p\eta_n)*(f+\kappa f^2)]-\text{div}_p[(p\eta_n)*(f+\kappa f^2)]+\kappa p\cdot \nabla_x[f_n^2-\eta_n*f^2]+\kappa \text{div}_p[p(\eta_n * f^2-f_n^2)].
\end{equation}

%%%%%%%%%%%%%%%%%%%%%%%%%%%%%%%%%%%%%
\begin{comment}
\begin{lemma}

\end{lemma}
\begin{proof}
We multiply \eqref{eq.fn0} by $|p|^2\zeta_{\epsilon}$ and integrate by parts
\begin{align*}
\partial_t \int_{\mathbb{T}^d}\int_{\mathbb{R}^d} |p|^2\zeta_{\epsilon} f_n dpdx=&\int_{\mathbb{T}^d}\int_{\mathbb{R}^d}
(|p|^2\Delta_p \zeta_{\epsilon}+4p\cdot \nabla_p \zeta_{\epsilon}+2d \zeta_{\epsilon})f_n dpdx\\
&-\int_{\mathbb{T}^d}\int_{\mathbb{R}^d}
(p\cdot \nabla_p \zeta_{\epsilon}+2\zeta_{\epsilon})|p|^2\eta_n*(f+\kappa f^2) dpdx\\
&+\int_{\mathbb{T}^d}\int_{\mathbb{R}^d}
(p\cdot \nabla_p \zeta_{\epsilon}+2\zeta_{\epsilon})|p|^2(p\eta_n)*(f+\kappa f^2)dpdx.
\end{align*}
Since $|p|^2\Delta_p \zeta_{\epsilon},\,p\cdot \nabla_p \zeta_{\epsilon},$ and $ \zeta_{\epsilon}$ are bounded, we can estimate
\begin{align*}
 \int_{\mathbb{T}^d}&\int_{\mathbb{R}^d} |p|^2f_n(T) dpdx- \int_{\mathbb{T}^d}\int_{\mathbb{R}^d} |p|^2 f_n(0) dpdx\\
 &\leq C\int_0^T \int_{\mathbb{T}^d}\int_{\mathbb{R}^d}  f_n dpdxdt+C \int_0^T\int_{\mathbb{T}^d}\int_{\mathbb{R}^d}|p|^2(\eta_n+p\eta_n)*(f+\kappa f^2)  dpdxdt.
\end{align*} 
\end{proof}
\end{comment}
 %%%%%%%%%%%%%%%%%%%%%%%%%%%%%%%%%%%
 
 %%%%%%%%%%%%%%%%%%%%%%%%%%%%%%%%%%%%%
 The following lemma will be used to prove the mass conservation in Lemma \ref{l:mass}.
\begin{lemma} Let $f\in C([0,\infty), L^1\cap  L^{2}(\mathbb{T}^d\times \mathbb{R}^d))$ be a weak solution to \eqref{Eq} and $f_n$ be its mollifications.  Then for any $T>0$
\begin{equation}\label{lim}
\lim_{n\to \infty}||\mathrm{div}_x[(p\eta_n)*(f+\kappa f^2)]-\mathrm{div}_p[(p\eta_n)*(f+\kappa f^2)]||_{L^1((0,T)\times \mathbb{T}^d\times \mathbb{R}^d)}=0
\end{equation}
\end{lemma}
\begin{proof}
 We denote $F\colonequals f+\kappa f^2 $ and  compute
 \begin{small}
\begin{align}\label{g_n1}
&||\text{div}_x[(p\eta_n)*F]-\text{div}_p[(p\eta_n)*F]||_{L^1((0,T)\times \mathbb{T}^d\times \mathbb{R}^d)}\nonumber\\&=\int_0^T\int_{\mathbb{T}^d}\int_{\mathbb{R}^d}\left|\tau_n *_t\int_{\mathbb{T}^d}\int_{\mathbb{R}^d}(\text{div}_x-\text{div}_p)[(p-p')\tilde{r}_n(x-x')r_n(p-p')](F(\cdot ,x',p')-F(\cdot,x,p))dp'dx'\right|dpdxdt\nonumber \\
&\leq C\int_0^T\left[\tau_n*_t\left(\sup_{|y|\leq 1/n, \,|q|\leq 1/n}\int_{\mathbb{T}^d}\int_{\mathbb{R}^d}|F(\cdot,x+y,p+q)-F(\cdot,x,p)|dpdx\right)\right](t)dt,
\end{align}\end{small}
where $C>0$ is a constant such that 
$$\int_{\mathbb{T}^d}\int_{\mathbb{R}^d}|(\text{div}_x-\text{div}_p)[(p-p')\tilde{r}_n(x-x')r_n(p-p')]|dp'dx'\leq C.$$
It is important to note that $C$ is independent of $t,x,p,n.$ Let us denote $$g_n(t)\colonequals \sup_{|y|\leq 1/n, \,|q|\leq 1/n}\int_{\mathbb{T}^d}\int_{\mathbb{R}^d}|F(t,x+y,p+q)-F(t,x,p)|dpdx,\,\,\,\,t\in [0,T+1/n].$$
Since $\tau_n\geq 0$ has its support in $(-1/n,0),$ we have
\begin{align*}
\int_0^T(\tau_n*_tg_n)(t)dt=&\int_0^T\left(\int_t^{t+1/n}\tau_n(t-s)g_n(s)ds\right)dt\\
=&\int_0^{1/n}\left(\int_0^{s}\tau_n(t-s)dt\right)g_n(s)ds+\int_{1/n}^{T}\left(\int_{s-1/n}^s\tau_n(t-s)dt\right)g_n(s)ds\\
&+\int_T^{T+1/n}\left(\int_{s-1/n}^{T}\tau_n(t-s)dt\right)g_n(s)ds.
\end{align*}
One can check that 
$$\int_0^{s}\tau_n(t-s)dt=\int_{-s}^{0}\tau_n(t)dt\leq \int_{-1/n}^{0}\tau_n(t)dt=1\,\,\,\,\, \text{for}\,\,\,s\in [0, 1/n],$$
$$\int_{s-1/n}^{s}\tau_n(t-s)dt= \int_{-1/n}^{0}\tau_n(t)dt=1\,\,\,\,\, \text{for}\,\,\,s\in [1/n,T],$$
$$\int_{s-1/n}^{T}\tau_n(t-s)dt= \int_{-1/n}^{T-s}\tau_n(t)dt\leq \int_{-1/n}^0\tau_n(t)dt=1\,\,\,\,\, \text{for}\,\,\,s\in [T, T+1/n].$$
These estimates show
 \begin{align}\label{g_n}
\int_0^T(\tau_n*_tg_n)(t)dt& \leq \int_0^{T+1/n}g_n(t)dt\nonumber \\
&=\sup_{|y| \leq 1/n, \,|q|\leq 1/n}\int_0^{T+1/n}\int_{\mathbb{T}^d}\int_{\mathbb{R}^d}|F(t,x+y,p+q)-F(t,x,p)|dpdxdt.
\end{align}
\eqref{g_n1} and \eqref{g_n}
show
\begin{align*}||\text{div}_x[(p\eta_n)*F]&-\text{div}_p[(p\eta_n)*F]||_{L^1((0,T)\times \mathbb{T}^d\times \mathbb{R}^d)}\\
\leq & C\sup_{|y| \leq 1/n, \,|q|\leq 1/n}\int_0^{T+1/n}\int_{\mathbb{T}^d}\int_{\mathbb{R}^d}|F(t,x+y,p+q)-F(t,x,p)|dpdxdt.
\end{align*}
Since translations are continuous in the $L^1((0,T+1)\times \mathbb{T}^d\times \mathbb{R}^d)$ norm (see \cite[Proposition 8.5]{Foll}), we get the claimed result.
\end{proof}
%%%%%%%%%%%%%%%%%%%%%%%%
{The following lemma will be used to prove the $L^1$-contraction of weak solutions in  Lemma \ref{max.pri} and to estimate the time derivative of the entropy functional in Lemma \ref{l:Ent}.}
\begin{lemma}\label{l:R_n}
Let $f\in C([0,\infty), L^1\cap  L^{2}(\mathbb{T}^d\times \mathbb{R}^d))$ be a weak solution to \eqref{Eq} and $f_n$ be its mollifications. Assume $f$ has regularity as in \eqref{reg}. Let $\psi \in L^{\infty}((0,\infty)\times \mathbb{T}^d\times \mathbb{R}^d)$ be such that $|p|\psi\in L^{\infty}((0,\infty)\times \mathbb{T}^d\times \mathbb{R}^d).$ Then for any $T>0$ 
\begin{equation}\label{com_p}
\lim_{n\to \infty}\int_0^T\int_{\mathbb{T}^d}
\int_{ \mathbb{R}^d} \psi (p\cdot \nabla_x[f_n^2-\eta_n*f^2]+ \mathrm{div}_p[p(\eta_n * f^2-f_n^2)])dpdxdt=0.
\end{equation} 
In particular, we have
\begin{equation}\label{com_p R}
\lim_{n\to \infty}\int_0^T\int_{\mathbb{T}^d}
\int_{ \mathbb{R}^d} \psi R_n(f)dpdxdt=0.
\end{equation} 
\end{lemma}
\begin{proof}
We consider the first term of \eqref{com_p}
$$ \psi p\cdot\nabla_x(f_n^2-\eta_n*f^2)=2 \psi f_n  p\cdot(\nabla_xf_n-\nabla_xf)+2 \psi p\cdot \nabla_xf (f_n-f)+ \psi p\cdot(\nabla_xf^2 -\nabla_x\eta_n*f^2).$$
Because of the regularity assumption \eqref{reg}, we have 
$f_n(\nabla_xf_n-\nabla_xf)\to 0$ and $\nabla_xf^2 -\nabla_x\eta_n*f^2\to 0$ in $ L^1_{loc}([0,\infty), L^1(\mathbb{T}^d\times \mathbb{R}^d))$ as $n\to \infty.$ 
As we have $\nabla_x f\in L^2_{loc}([0,\infty),L^2(\mathbb{T}^d\times \mathbb{R}^d))$ and $f_n-f\to 0$ in $ L^2_{loc}([0,\infty), L^2(\mathbb{T}^d\times \mathbb{R}^d))$ as $n\to \infty,$   we can  show by the H\"older inequality that  $\nabla_xf (f_n-f)\to 0$ in $ L^1_{loc}([0,\infty), L^1(\mathbb{T}^d\times \mathbb{R}^d))$ as $n\to \infty.$ Using these limits and the fact that $\psi$ and  $|p|\psi$ are bounded, we obtain
 \begin{equation*}\lim_{n\to \infty}\int^T_0\int_{\mathbb{T}^d} \int_{\mathbb{R}^d}\psi p\cdot \nabla_x [f_n^2-\eta_n*f^2]dpdxdt=0.
 \end{equation*}
 The second term of \eqref{com_p} can be written as 
  $$\psi \mathrm{div}_p[p(\eta_n*f^2-f^2_n)]=d \psi (\eta_n*f^2-f^2_n)+\psi p\cdot \nabla_p(\eta_n*f^2-f^2_n),$$
  and similar arguments show 
  \begin{equation*}\lim_{n\to \infty}\int^T_0\int_{\mathbb{T}^d} \int_{\mathbb{R}^d}\psi \mathrm{div}_p[p(\eta_n*f^2-f^2_n)] dpdxdt=0.
  \end{equation*}
  The proof of \eqref{com_p R} follows from \eqref{lim} and \eqref{com_p}.
\end{proof}
%%%%%%%%%%%%%%%%%%%%%%%%%
 \subsection{Mass conservation}
%%%%%%%%%%%%%%%%%%%%%%
To verify the mass conservation of weak solutions we introduce a cut-off in the $p-$direction. Let $\zeta:[0,\infty)\to \mathbb{R}$ be a non-increasing function in $ C^{\infty}((0,\infty))$ such that $\zeta(s)=1$ for $s\in [0,1]$ and $\zeta(s)=0$ for $s\geq 2.$  We define
 \begin{equation}\label{zeta}
 \zeta_{\epsilon}(p)\colonequals \zeta \left(\epsilon |p|\right) ,\, \, p\in \mathbb{R}^d, \, \epsilon>0.
 \end{equation}
 We note that $\zeta_{\epsilon}(p)=1$ for $|p|\leq 1/\epsilon$ and $\zeta_{\epsilon}(p)=0$ for $|p|\geq 2/\epsilon.$ There is a constant $C>0$ independent of $\epsilon$ such that 
 \begin{equation}\label{gr.z}
 |\nabla_{p}\zeta_{\epsilon}(p)|\leq C\epsilon, \, \, \, |\Delta_p \zeta_{\epsilon}(p)|\leq C\epsilon^2, \, \, \, |\nabla_{p}\zeta_{\epsilon}(p)\cdot p|\leq C, \, \, \, |p|^2|\Delta_p \zeta_{\epsilon}(p)|\leq C
 \end{equation}
  for all  $p\in \mathbb{R}^d.$
 %%%%%%%%%%%%%%%%%%%%%%%%%%%%%%%%
 %%%%%%%%%%%%%%%%%%%%%
\begin{lemma}\label{l:mass}
Let $f\in C([0,\infty), L^1\cap  L^{2}(\mathbb{T}^d\times \mathbb{R}^d))$ be a weak solution to \eqref{Eq} with an initial condition $f_0.$   Then we have mass conservation
\begin{equation*}
\int_{\mathbb{T}^d}\int_{\mathbb{R}^d}f(t,x,p)dpdx=\int_{\mathbb{T}^d}\int_{\mathbb{R}^d}f_0(x,p)dpdx, \, \, \, \, \forall\, t\geq 0.
\end{equation*}
\end{lemma}
\begin{proof}
We multiply \eqref{eq.fn0} by $\zeta_{\epsilon}(p)$ and integrate by parts
\begin{align*}
\frac{d}{dt}\int_{\mathbb{T}^d}\int_{\mathbb{R}^d}\zeta_{\epsilon}f_ndpdx=&\int_{\mathbb{T}^d}\int_{\mathbb{R}^d} (\Delta_p \zeta_{\epsilon} f_n-\nabla_p \zeta_{\epsilon}\cdot p\eta_n*(f+\kappa f^2))dpdx\\
&+\int_{\mathbb{T}^d}\int_{\mathbb{R}^d}\zeta_{\epsilon}(\mathrm{div}_x[(p\eta_n)*(f+\kappa f^2)]-\mathrm{div}_p[(p\eta_n)*(f+\kappa f^2)])dpdx.
\end{align*}
We integrate this equation on the interval $[0,T]$ for any $T>0:$
\begin{align*}
\int_{\mathbb{T}^d}\int_{\mathbb{R}^d}\zeta_{\epsilon}f_n(T)dpdx&-\int_{\mathbb{T}^d}\int_{\mathbb{R}^d}\zeta_{\epsilon}f_n(0)dpdx\\=&\int_0^T\int_{\mathbb{T}^d}\int_{\mathbb{R}^d} (\Delta_p \zeta_{\epsilon} f_n-\nabla_p \zeta_{\epsilon}\cdot p\eta_n*(f+\kappa f^2))dpdxdt\\
&+\int_0^T\int_{\mathbb{T}^d}\int_{\mathbb{R}^d}\zeta_{\epsilon}(\mathrm{div}_x[(p\eta_n)*(f+\kappa f^2)]-\mathrm{div}_p[(p\eta_n)*(f+\kappa f^2)])dpdxdt.
\end{align*}
 We keep here the $\mathrm{div}_x-$term, as we shall use \eqref{lim} later on. In this equation we first let $n\to \infty$ and then $\epsilon\to 0:$ the fact that  $\zeta_{\epsilon}$ is bounded  and  the Lebesgue dominated convergence theorem  show
 $$\lim_{\epsilon\to 0}\lim_{n\to \infty}\big(\int_{\mathbb{T}^d}\int_{\mathbb{R}^d}\zeta_{\epsilon}f_n(T)dpdx-\int_{\mathbb{T}^d}\int_{\mathbb{R}^d}\zeta_{\epsilon}f_n(0)dpdx\big)=\int_{\mathbb{T}^d}\int_{\mathbb{R}^d}f(T)dpdx-\int_{\mathbb{T}^d}\int_{\mathbb{R}^d}f_0dpdx.$$ 
 Since $\Delta_p \zeta_{\epsilon}$ and $\nabla_p \zeta_{\epsilon}\cdot p$ are uniformly  bounded in $\epsilon$, we have 
\begin{multline*}\lim_{n\to \infty} \int_0^T\int_{\mathbb{T}^d}\int_{\mathbb{R}^d} (\Delta_p \zeta_{\epsilon} f_n-\nabla_p \zeta_{\epsilon}\cdot p\eta_n*(f+\kappa f^2))dpdxdt\\=\int_0^T\int_{\mathbb{T}^d}\int_{\mathbb{R}^d} (\Delta_p \zeta_{\epsilon} f-\nabla_p \zeta_{\epsilon}\cdot p(f+\kappa f^2))dpdxdt.
\end{multline*}
{The term $\Delta_p \zeta_{\epsilon} f-\nabla_p \zeta_{\epsilon}\cdot p(f+\kappa f^2)$ converges  pointwise to zero  as $\epsilon \to 0.$} Moreover, $\Delta_p \zeta_{\epsilon} $ and $\nabla_p \zeta_{\epsilon}\cdot p$ are bounded. Hence, we can apply the Lebesgue dominated convergence theorem 
 $$\lim_{\epsilon\to 0}\int_0^T\int_{\mathbb{T}^d}\int_{\mathbb{R}^d} (\Delta_p \zeta_{\epsilon} f-\nabla_p \zeta_{\epsilon}\cdot p(f+\kappa f^2))dpdxdt=0.$$
The fact that $\zeta_{\epsilon}$ is bounded and \eqref{lim} imply
$$
\lim_{\epsilon\to 0}\lim_{n\to \infty} \int_0^T\int_{\mathbb{T}^d}\int_{\mathbb{R}^d}\zeta_{\epsilon}(\mathrm{div}_x[(p\eta_n)*(f+\kappa f^2)]-\mathrm{div}_p[(p\eta_n)*(f+\kappa f^2)])dpdxdt=0.
$$
 These limits show
 $$\int_{\mathbb{T}^d}\int_{\mathbb{R}^d}f(T)dpdx=\int_{\mathbb{T}^d}\int_{\mathbb{R}^d}f_0dpdx$$
for all $T>0.$
\end{proof}

%%%%%%%%%%%%%%%%%%%
\begin{Remark}
We note that the mass conservation in Lemma \ref{l:mass} does not require the regularity assumption in  \eqref{reg}. However, we will {need} this regularity assumption to prove the $L^1$-contraction of weak solutions in Lemma \ref{max.pri} and  to estimate  the time derivative of the entropy functional in Lemma \ref{l:Ent}. 
\end{Remark}
%%%%%%%%%%%%%%%%%%%%%%%%%%%%%%%%%%%%%%%%%%%%%%%%%%%%%%%%%%%%%%%%%%%%%%%
\subsection{Contraction of the $L^1 $ norm and maximum principle}
\begin{lemma}\label{max.pri}
Let $f,\, g\in C([0,\infty), L^1\cap  L^{2}(\mathbb{T}^d\times \mathbb{R}^d))$  be two weak solutions of \eqref{Eq} with  initial conditions $f_0$ and $g_0,$ respectively. Assume $f$ and $g$ have regularity as in \eqref{reg}.
 Then
\begin{equation*}
||f(t)-g(t)||_{L^1(\mathbb{T}^d\times \mathbb{R}^{d})}\leq ||f_0-g_0||_{L^1(\mathbb{T}^d\times \mathbb{R}^{d})}
\end{equation*} 
for all $t\geq 0.$ 
\end{lemma}
\begin{proof}
For all $n\in \mathbb{N},$ we define
$$\text{sign}_{\varepsilon}(s)\colonequals \begin{cases} -1,\, \, \,\, \, \,  &\text{if}\,\, \,s\leq -\varepsilon\\
\eta_{\varepsilon}(s),\, \, \, &\text{if}\, \, \, -\varepsilon\leq s\leq  \varepsilon\\
1, \, \, \, \, \, \, \,\, \,  \, \, &\text{if}\, \, \, s\geq \varepsilon
\end{cases} $$
with an increasing odd function $\eta_{\varepsilon}\in C^{\infty}([-\varepsilon,\varepsilon])$ such that $\text{sign}_{\varepsilon}$ is $ C^{\infty}$ at $s=\pm \varepsilon.$   Let $|\cdot|_{\varepsilon}$ denote the primitive of $\text{sign}_{\varepsilon}$ which vanishes at 0. One can  check that  $\text{sign}_{\varepsilon}$ and $|\cdot|_{\varepsilon}$ converge pointwise  to the usual sign  and absolute value functions as $\varepsilon\to 0,$ respectively.

 Let $f_n$ and $g_n$ be defined as in \eqref{f_n}. Then, by \eqref{eq.fn} we have  
\begin{equation*}
\partial_t(f_n-g_n)+p\cdot \nabla_x(f_n-g_n+\kappa (f^2_n-g^2_n))-\Delta_p(f_n-g_n)-\text{div}_p(p(f_n-g_n) +p\kappa (f^2_n-g^2_n))=R_n(f)-R_n(g). 
\end{equation*}
 Let $\zeta_{\epsilon}$ be the function defined in \eqref{zeta}. We multiply this equation by $\zeta_{\epsilon}(p) \text{sign}_{\varepsilon}(f_n-g_n)$ and integrate by parts
 \begin{align*}
\frac{d}{dt}\int_{\mathbb{T}^d}& \int_{\mathbb{R}^d}\zeta_{\epsilon}(p)|f_n-g_n|_{\varepsilon} dpdx\\=&-\int_{\mathbb{T}^d} \int_{\mathbb{R}^d}\zeta_{\epsilon}(p)  \text{sign}_{\varepsilon}'(f_n-g_n)|\nabla_p(f_n-g_n)|^2dpdx\\&-\int_{\mathbb{T}^d} \int_{\mathbb{R}^d}\text{sign}_{\varepsilon}(f_n-g_n) \nabla_p\zeta_{\epsilon}(p)\cdot (\nabla_p f_n+p(f_n+\kappa f^2_n)- \nabla_p g_n-p(g_n+\kappa g^2_n))dpdx\\
&-\int_{\mathbb{T}^d} \int_{\mathbb{R}^d}\zeta_{\epsilon}(p) \text{sign}_{\varepsilon}'(f_n-g_n)p\cdot\nabla_p(f_n-g_n)(f_n-g_n)(1+\kappa f_n+\kappa g_n)dpdx\\
&+\int_{\mathbb{T}^d} \int_{\mathbb{R}^d}\zeta_{\epsilon}(p) \text{sign}_{\varepsilon}'(f_n-g_n)p\cdot\nabla_x(f_n-g_n)(f_n-g_n)(1+\kappa f_n+\kappa g_n)dpdx\\
&+\int_{\mathbb{T}^d} \int_{\mathbb{R}^d}\zeta_{\epsilon}(p) \text{sign}_{\varepsilon}(f_n-g_n)(R_n(f)-R_n(g))dpdx.
 \end{align*}
 Note that this formula includes both $\varepsilon$ and $\epsilon.$ We integrate this equation on the interval $[0,T],\, T>0,$ and  use
 $\zeta_{\epsilon}(p) \text{sign}_{\varepsilon}'(f_n-g_n)|\nabla_p(f_n-g_n)|^2\geq 0$ to estimate
  \begin{align*}
\int_{\mathbb{T}^d}  &\int_{\mathbb{R}^d}\zeta_{\epsilon}(p)|f_n(T)-g_n(T)|_{\varepsilon} dpdx-\int_{\mathbb{T}^d}  \int_{\mathbb{R}^d}\zeta_{\epsilon}(p)|f_n(0)-g_n(0)|_{\varepsilon} dpdx\\ \leq &-\int_0^T\int_{\mathbb{T}^d} \int_{\mathbb{R}^d}\text{sign}_{\varepsilon}(f_n-g_n) \nabla_p\zeta_{\epsilon}(p)\cdot (\nabla_p f_n+p(f_n+\kappa f^2_n)- \nabla_p g_n-p(g_n+\kappa g^2_n))dpdxdt\\
&-\int_0^T\int_{\mathbb{T}^d} \int_{\mathbb{R}^d}\zeta_{\epsilon}(p) \text{sign}_{\varepsilon}'(f_n-g_n)p\cdot\nabla_p(f_n-g_n)(f_n-g_n)(1+\kappa f_n+\kappa g_n)dpdxdt\\
&+\int_0^T\int_{\mathbb{T}^d} \int_{\mathbb{R}^d}\zeta_{\epsilon}(p) \text{sign}_{\varepsilon}'(f_n-g_n)p\cdot\nabla_x(f_n-g_n)(f_n-g_n)(1+\kappa f_n+\kappa g_n)dpdxdt\\
& +\int_0^T\int_{\mathbb{T}^d} \int_{\mathbb{R}^d}\zeta_{\epsilon}(p) \text{sign}_{\varepsilon}(f_n-g_n)(R_n(f)-R_n(g))dpdxdt.
 \end{align*}
 Using the identities
 $$\text{sign}_{\varepsilon}'(f_n-g_n)\nabla_x(f_n-g_n)(f_n-g_n)=\nabla_x((f_n-g_n)\text{sign}_{\varepsilon}(f_n-g_n)-|f_n-g_n|_{\varepsilon}),$$
 $$\text{sign}_{\varepsilon}'(f_n-g_n)\nabla_p(f_n-g_n)(f_n-g_n)=\nabla_p((f_n-g_n)\text{sign}_{\varepsilon}(f_n-g_n)-|f_n-g_n|_{\varepsilon}),$$
 we integrate by parts
 \begin{align*}
\int_{\mathbb{T}^d}  &\int_{\mathbb{R}^d}\zeta_{\epsilon}(p)|f_n(T)-g_n(T)|_{\varepsilon} dpdx-\int_{\mathbb{T}^d}  \int_{\mathbb{R}^d}\zeta_{\epsilon}(p)|f_n(0)-g_n(0)|_{\varepsilon} dpdx\\ \leq &-\int_0^T\int_{\mathbb{T}^d} \int_{\mathbb{R}^d}\text{sign}_{\varepsilon}(f_n-g_n) \nabla_p\zeta_{\epsilon}(p)\cdot (\nabla_p f_n+p(f_n+\kappa f^2_n)- \nabla_p g_n-p(g_n+\kappa g^2_n))dpdxdt\\
&+\int_0^T\int_{\mathbb{T}^d} \int_{\mathbb{R}^d}\text{div}_p[\zeta_{\epsilon}(p)(1+\kappa f_n+\kappa g_n)p] ((f_n-g_n)\text{sign}_{\varepsilon}(f_n-g_n)-|f_n-g_n|_{\varepsilon})dpdxdt\\
&-\int_0^T\int_{\mathbb{T}^d} \int_{\mathbb{R}^d}\kappa\zeta_{\epsilon}(p) p\cdot \nabla_x(f_n+ g_n)((f_n-g_n)\text{sign}_{\varepsilon}(f_n-g_n)-|f_n-g_n|_{\varepsilon}) dpdxdt\\
&+\int_0^T\int_{\mathbb{T}^d} \int_{\mathbb{R}^d}\zeta_{\epsilon}(p) \text{sign}_{\varepsilon}(f_n-g_n)(R_n(f)-R_n(g))dpdxdt.
 \end{align*}
 Since $|f_n-g_n|_{\varepsilon}\to |f_n-g_n|,$ $\text{sign}_{\varepsilon}(f_n-g_n)\to \text{sign}(f_n-g_n),$ and $(f_n-g_n)\text{sign}_{\varepsilon}(f_n-g_n)-|f_n-g_n|_{\varepsilon}\to 0$ pointwise  as $\varepsilon\to 0,$ we conclude that 
\begin{align*}
\int_{\mathbb{T}^d}& \int_{\mathbb{R}^d}\zeta_{\epsilon}(p)|f_n(T)-g_n(T)| dpdx-\int_{\mathbb{T}^d} \int_{\mathbb{R}^d}\zeta_{\epsilon}(p)|f_n(0)-g_n(0)| dpdx\nonumber \\ \leq &-\int_0^T\int_{\mathbb{T}^d} \int_{\mathbb{R}^d}\text{sign}(f_n-g_n) \nabla_p\zeta_{\epsilon}(p)\cdot (\nabla_p f_n+p(f_n+\kappa f^2_n)- \nabla_p g_n-p(g_n+\kappa g^2_n))dpdxdt \nonumber \\
&+\int^T_0\int_{\mathbb{T}^d} \int_{\mathbb{R}^d}\zeta_{\epsilon}(p) \text{sign}(f_n-g_n)(R_n(f)-R_n(g))dpdxdt.
 \end{align*}
  Using the regularity assumptions \eqref{reg} and the fact that $\nabla_{p}\zeta_{\epsilon}(p)\cdot p$ and $\text{sign}(f_n-g_n)$ are bounded for a fixed $\epsilon,$ we can pass to the limit as  $n\to \infty$
\begin{align}\label{aXl}
\int_{\mathbb{T}^d}& \int_{\mathbb{R}^d}\zeta_{\epsilon}(p)|f(T)-g(T)| dpdx-\int_{\mathbb{T}^d} \int_{\mathbb{R}^d}\zeta_{\epsilon}(p)|f_0-g_0| dpdx\nonumber \\ \leq &\int_0^T\int_{\mathbb{T}^d} \int_{\mathbb{R}^d}\big| \nabla_p\zeta_{\epsilon}(p)\cdot (\nabla_p f+p(f+\kappa f^2)- \nabla_p g-p(g+\kappa g^2))\big|dpdxdt \nonumber \\
&+\lim_{n\to \infty}\int^T_0\int_{\mathbb{T}^d} \int_{\mathbb{R}^d}\zeta_{\epsilon}(p) \text{sign}(f_n-g_n)(R_n(f)-R_n(g))dpdxdt.
 \end{align}
  Since $|\text{sign}(f_n-g_n)|\leq 1,$ one can check  that \eqref{com_p R} with $\psi\colonequals \zeta_{\epsilon}(p) \text{sign}(f_n-g_n)$ still holds, even if here $\psi $ depends on $n.$ This provides
   \begin{equation*}\lim_{n\to \infty}\int^T_0\int_{\mathbb{T}^d} \int_{\mathbb{R}^d}\zeta_{\epsilon}(p) \text{sign}(f_n-g_n)(R_n(f)-R_n(g))dpdx=0.
 \end{equation*}
 Then we  let $\epsilon\to 0$ in \eqref{aXl}: 
 \begin{align*}
\int_{\mathbb{T}^d}& \int_{\mathbb{R}^d}|f(T)-g(T)| dpdx-\int_{\mathbb{T}^d} \int_{\mathbb{R}^d}|f_0-g_0| dpdx\nonumber \\ \leq &\lim_{\epsilon\to 0}\int_0^T\int_{\mathbb{T}^d} \int_{\mathbb{R}^d}\big| \nabla_p\zeta_{\epsilon}(p)\cdot (\nabla_p f+p(f+\kappa f^2)- \nabla_p g-p(g+\kappa g^2))\big|dpdxdt.
 \end{align*}
 It is easy to check that $ \nabla_p\zeta_{\epsilon}(p)\cdot (\nabla_p f+p(f+\kappa f^2)- \nabla_p g-p(g+\kappa g^2))$ converges pointwise to zero  because of \eqref{gr.z}. Moreover,  $|\nabla_p \zeta_{\epsilon}(p)|$ and $\nabla_p \zeta_{\epsilon}(p)\cdot p$ (terms which depend on $\epsilon$) are bounded.  The regularity assumptions \eqref{reg} imply  $|\nabla_p f|$, $|\nabla_p g|,$ $ f+\kappa f^2,$ and $ g+\kappa g^2$ are in $ L^1_{loc}([0,\infty),L^1(\mathbb{T}^d\times \mathbb{R}^d)).$ Hence, we can use  the Lebesgue dominated convergence  theorem to pass to the limit 
  and obtain 
 $$\int_{\mathbb{T}^d} \int_{\mathbb{R}^d}|f(T)-g(T)| dpdx\leq \int_{\mathbb{T}^d} \int_{\mathbb{R}^d}|f_0-g_0| dpdx.$$ Since $T>0$ is arbitrary, we obtain the  claimed result. 
\end{proof}

{We note that Lemma \ref{max.pri} implies uniqueness of weak solutions to \eqref{Eq} that have regularity as in \eqref{reg}.}
%%%%%%%%%%%%%%%%
\begin{corollary}\label{Cor:bound}
Let $f,\,g \in C([0,\infty), L^1\cap  L^{2}(\mathbb{T}^d\times \mathbb{R}^d))$  be two weak solutions of \eqref{Eq} with  initial conditions $f_0$ and $g_0,$ respectively. Assume $f$ and $g$ have regularity as in \eqref{reg}. If $$f_0(x,p)\leq g_0(x,p)$$ for almost every $(x, p)\in\mathbb{T}^d\times \mathbb{R}^{d}, $  then $$f(t,x,p)\leq g(t,x,p)$$ for almost every $(x, p)\in\mathbb{T}^d\times \mathbb{R}^{d} $ and  for all $t\geq 0.$ 

Moreover, assume there exist constants $\beta_->0$ and $\beta_+>0$ such that 
\begin{equation}\label{f_0bound}
\frac{\beta_{-} M(p)}{1-\kappa \beta_{-} M(p)}\leq f_0(x,p)\leq \frac{\beta_{+} M(p)}{1-\kappa \beta_{+} M(p)}, \, \, \, \forall\, x \in \mathbb{T}^d,\, \, \,  \forall\,p\in \mathbb{R}^d,
\end{equation}
where we require $\beta_{+}<(2\pi )^{d/2}$ when $\kappa=1.$
 Then  the solution $f$ satisfies the same bounds, i.e.,
\begin{equation}\label{bound}
\frac{\beta_{-} M(p)}{1-\kappa \beta_{-} M(p)}\leq f(t,x,p)\leq \frac{\beta_{+} M(p)}{1-\kappa \beta_{+} M(p)}
\end{equation}
for almost every $(x, p)\in\mathbb{T}^d\times \mathbb{R}^{d} $ and  for all $t\geq 0.$ 
\end{corollary}
\begin{proof} {The function $|f(t,x,p)-g(t,x,p)|+f(t,x,p)-g(t,x,p)$ is nonnegative almost everywhere. It is zero a.e.\,if  $f\leq g $ a.e. } 
 The {$L^1$-contractivity} and the mass conservation {of weak solutions} imply 
 \begin{align*}
 0&\leq \int_{\mathbb{T}^d} \int_{\mathbb{R}^d}(|f(t,x,p)-g(t,x,p)|+f(t,x,p)-g(t,x,p)) dpdx \\&\leq \int_{\mathbb{T}^d} \int_{\mathbb{R}^d}(|f_0(x,p)-g_0(x,p)|+f_0(x,p)-g_0(x,p)) dpdx\\
 &=\int_{\mathbb{T}^d} \int_{\mathbb{R}^d}(g_0-f_0+f_0-g_0) dpdx=0.
 \end{align*}
{We conclude that  $\int_{\mathbb{T}^d} \int_{\mathbb{R}^d}(|f(t,x,p)-g(t,x,p)|+f(t,x,p)-g(t,x,p)) dpdx$ is zero for all $t\geq 0.$}
 This shows $|f(t,x,p)-g(t,x,p)|+f(t,x,p)-g(t,x,p)= 0$ for a.e.  $(x,p)$ and  for all $t\geq 0,$ which means $f(t,x,p)\leq g(t,x,p)$ for a.e. $(x,p)$ and  for all $t\geq 0.$

The inequalities in  \eqref{bound} follow from the fact that {the two functions $\frac{\beta_{\pm} M(p)}{1-\kappa \beta_{\pm} M(p)}$ (which are constant in $x$ and in $t$)}  solve \eqref{Eq}, and {$f_0$, the initial data of $f$} satisfies \eqref{f_0bound}. 
\end{proof}
%%%%%%%%%%%%%%%%%%%%
\begin{Remark}
The lower and upper bounds in \eqref{bound} will play an important role in the next section. In particular, these bounds will be used to prove that the relative entropy functional in Lemma \ref{C<H<C} and the modified entropy functional in Lemma \ref{l: E eq} are equivalent to $||f-f_{\infty}||^2.$ 
\end{Remark}
%%%%%%%%%%%%%
\section{Lyapunov functional and exponential decay}
In this section we shall construct (in several steps) a Lyapunov functional for the decay analysis of \eqref{Eq}, i.e., for the proof of our main result, Theorem \ref{main th}.
\subsection{Projection operator}
We  denote the macroscopic densities
$$\rho_{\infty}\colonequals \int_{\mathbb{R}^d}f_{\infty}(p)dp$$
and
 $$\rho(t,x)\colonequals \int_{\mathbb{R}^d}f(t,x,p)dp.$$
We define the nonlinear projection operator\footnote{The word  \emph{projection} (or \emph{projection operator}) is mostly used for a continuous linear operator $P$ from a Banach or Hilbert space to itself that is idempotent, i.e., $P^2=P.$ Although $\Pi[f]$ is a nonlinear operator, we call it  \emph{projection} as it has the property $\Pi[\Pi[f]]=\Pi[f]$.} 
\begin{equation}\label{pr}
\Pi [f](t,x,p)\colonequals\frac{\beta(t,x)M(p)}{1-\kappa \beta(t,x)M(p)}, \, \, \, f(t,x,\cdot)\in L^2(\mathbb{R}^d, M^{-1}),
\end{equation}
where $\beta(t,x)$ is uniquely chosen by the condition
\begin{equation*}
\int_{\mathbb{R}^d}\Pi [f]dp=\rho(t,x).
\end{equation*}
When $\kappa=1$ and $d\geq 3,$ $\Pi[f]$ is well-defined for functions $ f(t,x,\cdot)\in L^2(\mathbb{R}^d, M^{-1})$ with  $\rho(t,x)<m_c,$ where $m_c$ is the critical mass defined in \eqref{m_c}.   
The operator $\Pi [f]$ represents the projection of  $f$ to the manifold of functions which are local equilibria for \eqref{Eq}. This means that $\Pi$ maps $f(t,x,\cdot)\in L^2(\mathbb{R}^d, M^{-1})$ to a function $\Pi[f](t,x,\cdot)\in L^2(\mathbb{R}^d, M^{-1})$ at which the integral on the right hand side of \eqref{dt H2} vanishes. 

 If $f$ satisfies the bounds in \eqref{bound} (with $\beta_+<(2\pi)^{d/2}$ when $\kappa=1$), then we have $\beta_{-}\leq \beta(t,x)\leq \beta_+$ whose proof follows directly from  the definition of $\Pi[f]$. Since $\Pi [f](t,x,p)=\Pi [f](t,x,-p)$ for all $p=(p_1,...,p_d)^T\in \mathbb{R}^d,$ we have 
\begin{equation}\label{pP=0}
\int_{\mathbb{R}^d}p_i \Pi [f]dp=\int_{\mathbb{R}^d}p_i\frac{\beta(t,x)M(p)}{1-\kappa \beta(t,x)M(p)}dp=0, \, \, \, i \in \{1,..., d\}.
\end{equation}
Using $\partial_{p_i}\Pi [f]=-p_i(\Pi [f]+\kappa (\Pi [f])^2)$ and integrating by parts,  we get
\begin{equation}\label{vP=0}
\int_{\mathbb{R}^d}p_i(\Pi [f]+\kappa (\Pi [f])^2)dp=-\int_{\mathbb{R}^d}\partial_{p_i}\Pi [f]dp=0,
\end{equation}
and 
\begin{equation}\label{v_iv_jP}
\int_{\mathbb{R}^d}p_ip_j(\Pi [f]+\kappa (\Pi [f])^2) dp=-\int_{\mathbb{R}^d}p_i\partial_{p_j}\Pi [f] dp=\delta_{ij}\rho,\, \, \, i,j\in\{1,...,d\},
\end{equation}
where  $\delta_{ij}=\begin{cases}1,
\,\,\, \text{if}\,\, i=j\\
0, \,\,\, \text{if}\,\, i\neq j
\end{cases}$ is the Kronecker delta function. 

\begin{lemma} 
Let $f\in C([0,\infty), L^1\cap  L^{2}(\mathbb{T}^d\times \mathbb{R}^d))$  be a weak solution of \eqref{Eq} with  initial condition $f_0$ satisfying  
\begin{equation*}
{0\leq  f_0(x,p)} \leq \frac{\beta_{+} M(p)}{1-\kappa \beta_{+} M(p)} 
\end{equation*}
for some {constant} $\beta_+>0$ and for almost every $(x, p)\in\mathbb{T}^d\times \mathbb{R}^{d} $.  We also require $\beta_{+}<(2\pi )^{d/2}$ when $\kappa=1.$ Assume $f$ has regularity as in \eqref{reg}.
 Then there are constants $C_1>0$ and  $C_2>0$ (depending only on  $\beta_+$)  such that \begin{equation}\label{rr} C_1\frac{(\Pi [f]-f_{\infty})^2}{M^2}\leq (\rho-\rho_{\infty})^2\leq C_2\frac{(\Pi [f]-f_{\infty})^2}{M^2}.
\end{equation}
\end{lemma}
\begin{proof}
 By the definition of $\Pi$ and \eqref{eqilib} we have  
$$\beta-\beta_{\infty}=\frac{\Pi [f]-f_{\infty}}{M(1+\kappa\Pi [f])(1+\kappa f_{\infty})}$$
and $$\rho-\rho_{\infty}=\int_{\mathbb{R}^d}(\Pi [f]-f_{\infty})dp=(\beta-\beta_{\infty})\int_{\mathbb{R}^d}\frac{Mdp}{(1-\kappa \beta M)(1-\kappa \beta_{\infty}M)}.$$
These equations imply 
$$(\rho-\rho_{\infty})^2=\frac{(\Pi [f]-f_{\infty})^2}{M^2(1+\kappa \Pi [f])^2(1+\kappa f_{\infty})^2}\left(\int_{\mathbb{R}^d}\frac{M dp}{(1-\kappa \beta M)(1-\kappa \beta_{\infty}M)}\right)^2.$$
 Corollary \ref{Cor:bound} implies that  $f$ satisfies \begin{equation*}
0\leq  f(t, x,p)\leq \frac{\beta_{+} M(p)}{1-\kappa \beta_{+} M(p)} 
\end{equation*}
for almost every $(x, p)\in\mathbb{T}^d\times \mathbb{R}^{d} $ and for all $t\geq 0.$ From this bound we conclude  $0\leq  \beta(t,x)\leq \beta_{+}$ for all $t\geq 0$ and almost every $x\in \mathbb{T}^d$. It lets us estimate
$$1<1+\beta(t,x) M(p)\leq 1+\beta_+M(0)$$
if $\kappa=-1$, and $$1-\beta_+M(0)\leq 1-\beta(t,x) M(p)<1$$
if $\kappa=1$.   Note here $1-\beta_+M(0)$ is  positive  as we assume $\beta_+<(2\pi)^{d/2}.$ These inequalities {also} hold if we replace $\beta(t,x) $ with $\beta_{\infty}.$
 Therefore, $1+\kappa \Pi [f]=\frac{1}{1-\kappa \beta M}$ and  $1+\kappa f_{\infty}=\frac{1}{1-\kappa \beta_{\infty} M}$  are bounded from below and above by positive constants.  The same claim holds for  $$\displaystyle \int_{\mathbb{R}^d}\frac{M dp}{(1-\kappa \beta M)(1-\kappa \beta_{\infty}M)}.$$ This implies  the claimed inequalities.
\end{proof}

%%%%%%%%%%%%%%%%%%%%%%%%%%%%%%%%%%%%%%%%%%%%%%%%%%%%%%%%%%%%%%%%%%%%%%%%
\subsection{Relative entropy}
For non-negative functions $f$ and $g$ on $\mathbb{T}^d\times \mathbb{R}^d,$ we define the relative entropy functional 
\begin{equation}\label{Hch}
\mathrm{H}[f|g]\colonequals  \int_{\mathbb{T}^d}\int_{\mathbb{R}^d} \left(f \ln{\frac{f}{g}}-\kappa(1+\kappa f)\ln{\frac{1+\kappa f}{1+\kappa g}}\right)dpdx.
\end{equation}
We note, if $f$ and $f_{\infty}$ have the same mass, then $$\mathrm{H}[f|f_{\infty}]=  \mathrm{H}[f]-\mathrm{H}[f_{\infty}].$$ 
In the following two lemmata we study some properties of this functional, namely its equivalence to the squared $L^2$-difference of two functions and the time decay of the relative entropy.
%%%%%%%%%%%
\begin{lemma}\label{C<H<C}
Let $f$ and $g$ be two nonnegative {(possibly time-independent)} functions  satisfying \eqref{bound}. Then there are constants $C_3>0$ and $C_4>0$ (depending only on $\beta_-$ and $\beta_+$) such that
\begin{equation}\label{C_1<C_2}
C_3||f-g||^2\leq \mathrm{H}[f|g]\leq C_4||f-g||^2.
\end{equation}
Moreover, it holds that $f,g\in L^2(\mathbb{T}^d\times \mathbb{R}^d, M^{-1}).$
\end{lemma}
\begin{proof}
 We consider the integrand in \eqref{Hch}: By its Taylor expansion in the variable $f$ about $g$ with Lagrange reminder term, we obtain 
$$f \ln{\frac{f}{g}}-\kappa(1+\kappa f)\ln{\frac{1+\kappa f}{1+\kappa g}}=\frac{(f-g)^2}{2\xi(1+\kappa \xi)},$$
with a function $\xi=\xi(x,p)$ lying between $f$ and $g.$
Due to the assumptions on $f$ and $g,$ {also} $\xi$ satisfies \eqref{bound}. {As a consequence, when $\kappa=-1,$ we have {using also $\xi<1$}
$$2 \xi(1- \xi)\geq \frac{2\beta_{-}M(p)}{(1+\beta_{-}M(p))(1+\beta_{+}M(p))}\geq \frac{2\beta_{-}}{(1+\beta_{-}M(0))(1+\beta_{+}M(0))} M(p)$$
and 
$$2 \xi(1- \xi)\leq 2\xi \leq \frac{2\beta_{+}M(p)}{1+\beta_{+}M(p)}\leq 2\beta_{+} M(p).$$
When $\kappa=1,$ we have
$$2 \xi(1+\xi) 
{\geq \frac{2\beta_{-}M(p)}{(1-\beta_{-}M(p))^2}
\geq 2\beta_{-} M(p)}$$
and 
$$2 \xi(1+ \xi) 
{\leq \frac{2\beta_{+}M(p)}{(1-\beta_{+}M(p))^2}
\leq \frac{2\beta_{+}}{(1-\beta_{+}M(0))^2}M(p). }
$$
These estimates show that \eqref{C_1<C_2} holds with 
{
$$C_3\colonequals \begin{cases}  \frac{1}{2\beta_+}, \,\,\,\,&\text{if}\,\,\, \kappa=-1\\
\frac{(1-\beta_{+}M(0))^2}{2\beta_{+}},\,\, &\text{if}\,\,\, \kappa=1\end{cases},$$} 
and $$C_4\colonequals \begin{cases} \frac{(1+\beta_{-}M(0))(1+\beta_{+}M(0))}{2\beta_{-}}, \,\,\,\,&\text{if}\,\,\, \kappa=-1\\
{\frac{1}{2\beta_{-}}},\,\, &\text{if}\,\,\, \kappa=1\end{cases}.$$ }
%$$\frac{M}{C}\leq 2 \xi(1+\kappa \xi)\leq \frac{M}{C'}$$
%holds with appropriate positive constants $C$ and $ C',$ completing the proof.
\end{proof}
%%%%%%%%%%%%%%%%%%%%%%%%%%%%%%%%%%%%%%%
\begin{lemma}\label{l:Ent}
Let $f\in C([0,\infty), L^1\cap  L^{2}(\mathbb{T}^d\times \mathbb{R}^d))$ be a weak solution to \eqref{Eq} with the initial condition $f_0$ which satisfies \eqref{f_0bound}.
 Assume $f$ has regularity as in \eqref{reg}, and let $f_\infty$ be the equilibrium with the same mass.
 Then there is a constant $C_5>0$ such that
\begin{equation}\label{dt H<-C_5}
\mathrm{H}[f(T)|f_{\infty}]\leq \mathrm{H}[f_0|f_{\infty}] -C_5\int_0^T||f(t)-\Pi [f](t)||^2dt, \, \,\, \forall\, T>0.
\end{equation}
\end{lemma}
%%%%%%%%%%%%%%%%%
\begin{proof}
 Let $f_n$ and $\zeta_{\epsilon}\geq 0$ be defined as in \eqref{f_n} and \eqref{zeta}, respectively.
We consider $$\mathrm{H}_{\epsilon}[f_n|f_{\infty}]\colonequals \int_{\mathbb{T}^d}\int_{\mathbb{R}^d} \zeta_{\epsilon}(p)\left(f_n \ln{\frac{f_n}{f_{\infty}}}-\kappa(1+\kappa f_n)\ln{\frac{1+\kappa f_n}{1+\kappa f_{\infty}}}\right)dpdx.$$ 
This functional differs from $\mathrm{H}[f_n|f_{\infty}]$ by the appearance of the function $\zeta_{\epsilon}(p).$ It makes the integrand in $\mathrm{H}_{\epsilon}[f_n|f_{\infty}]$ have compact support with respect to $p,$ which will be  used to pass to the limit as $n\to \infty.$ Later we will show $\mathrm{H}_{\epsilon}[f|f_{\infty}] \to \mathrm{H}[f|f_{\infty}]$ as $\epsilon\to 0.$

\textbf{Step 1,} \textbf{$n\to \infty$:}
The assumptions of the lemma  provide that the solution $f$ satisfies the bounds \eqref{bound}. Hence, since $\eta_n\geq 0$ we have 
$$\eta_n *\frac{\beta_{-} M}{1-\kappa \beta_{-} M}\leq f_n\leq \eta_n *\frac{\beta_{+} M}{1-\kappa \beta_{+} M}, \, \, \, \forall\, x \in \mathbb{T}^d,\, \forall\, p\in \mathbb{R}^d,\,\, \forall\, t\geq 0.$$
By \cite[Theorem 8.14]{Foll} the sequence $\eta_n *\frac{\beta_{\pm} M}{1-\kappa \beta_{\pm} M}$ converges to $\frac{\beta_{\pm} M}{1-\kappa \beta_{\pm} M}$ uniformly {on $\R^d_p$} as $n\to \infty$. Hence, one can show that if $n$ is large enough, then {$f_n(t,x,p)$}  (as well as $1-{f_n(t,x,p)}$ when $\kappa=-1$) is bounded from below and above by positive constants (depending on $\epsilon$, $\beta_{-}$, and $\beta_{+}$) for all $x \in \mathbb{T}^d,$ $ p\in \text{supp}(\zeta_{\epsilon}),$ and $ t\geq 0.$
By using the proof of Lemma \ref{C<H<C} we can show that $\mathrm{H}_{\epsilon}[f_n|f_{\infty}]$ is equivalent to $$\displaystyle \int_{\mathbb{T}^d\times \mathbb{R}^d}\zeta_{\epsilon}\frac{(f_n-f_{\infty})^2}{M}dpdx.$$
Hence,  $\mathrm{H}_{\epsilon}[f_n|f_{\infty}]$ is well-defined and finite.

We use \eqref{eq.fn} to compute 
\begin{align*}
\frac{d}{dt}\mathrm{H}_{\epsilon}[f_n(t)|f_{\infty}]=&-\int_{\mathbb{T}^d}\int_{\mathbb{R}^d}\zeta_{\epsilon}(p)\frac{|\nabla_p f_n +p(f_n+\kappa f^2_n)|^2}{f_n(1+\kappa f_n)}dpdx\\
&-\int_{\mathbb{T}^d}\int_{\mathbb{R}^d}\nabla_p \zeta_{\epsilon}(p)\cdot (\nabla_p f_n +p(f_n+\kappa f^2_n))\ln\left({\frac{f_n}{\beta_{\infty}M(1+\kappa f_n)}}\right)dpdx\\
&+\int_{\mathbb{T}^d}\int_{\mathbb{R}^d} \zeta_{\epsilon}(p) R_n(f)\ln\left({\frac{f_n}{\beta_{\infty}M(1+\kappa f_n)}}\right)dpdx.
\end{align*}
We integrate this equation on the interval $[0,T],$ for any $T>0,$ 
\begin{align}\label{H_e}
\mathrm{H}_{\epsilon}[f_n(T)|f_{\infty}]&-\mathrm{H}_{\epsilon}[f_n(0)|f_{\infty}]\nonumber\\
=&-\int_0^T\int_{\mathbb{T}^d}\int_{\mathbb{R}^d}\zeta_{\epsilon}(p)\frac{|\nabla_p f_n +p(f_n+\kappa f^2_n)|^2}{f_n(1+\kappa f_n)}dpdxdt \nonumber \\
&-\int_0^T\int_{\mathbb{T}^d}\int_{\mathbb{R}^d}\nabla_p \zeta_{\epsilon}(p)\cdot (\nabla_p f_n +p(f_n+\kappa f^2_n))\ln\left({\frac{f_n}{\beta_{\infty}M(1+\kappa f_n)}}\right)dpdxdt \nonumber\\
&+\int_0^T\int_{\mathbb{T}^d}\int_{\mathbb{R}^d} \zeta_{\epsilon}(p) R_n(f)\ln\left({\frac{f_n}{\beta_{\infty}M(1+\kappa f_n)}}\right)dpdxdt.
\end{align}
The sequence $\zeta_{\epsilon}(p) \ln\left({\frac{f_n}{\beta_{\infty}M(1+\kappa f_n)}}\right)$ is uniformly bounded with respect to $n$ and has compact support with respect to $p$. This fact and  \eqref{com_p R} (e.g.\,with $\psi:=\chi_{\text{supp}(\zeta_\epsilon)}$) show
$$\lim_{n\to \infty}\int_0^T\int_{\mathbb{T}^d}\int_{\mathbb{R}^d} \zeta_{\epsilon}(p) R_n(f)\ln\left({\frac{f_n}{\beta_{\infty}M(1+\kappa f_n)}}\right)dpdxdt=0.$$
Next we write the second term on the r.h.s.\,of \eqref{H_e} as
\begin{align*}
&\int_0^T\int_{\mathbb{T}^d}\int_{\mathbb{R}^d}\nabla_p \zeta_{\epsilon}(p)\cdot (\nabla_p f_n +p(f_n+\kappa f^2_n))\ln\left({\frac{f_n}{\beta_{\infty}M(1+\kappa f_n)}}\right)
  dpdxdt\\
 =&\int_0^T\int_{\mathbb{T}^d}\int_{\mathbb{R}^d}\nabla_p \zeta_{\epsilon}(p)\cdot (\nabla_p (f_n-f) +p(f_n+\kappa f^2_n)-p(f+\kappa f^2))\ln\left({\frac{f_n}{\beta_{\infty}M(1+\kappa f_n)}}\right)
  dpdxdt\\
 & +\int_0^T\int_{\mathbb{T}^d}\int_{\mathbb{R}^d}\nabla_p \zeta_{\epsilon}(p)\cdot (\nabla_p f +p(f+\kappa f^2) )\ln\left({\frac{f_n}{\beta_{\infty}M(1+\kappa f_n)}}\right)
  dpdxdt\,.
 \end{align*}
 The terms 
 $\nabla_p \zeta_{\epsilon}(p)\ln\left({\frac{f_n}{\beta_{\infty}M(1+\kappa f_n)}}\right)$ and $\nabla_p \zeta_{\epsilon}(p)\cdot p \ln\left({\frac{f_n}{\beta_{\infty}M(1+\kappa f_n)}}\right) $ are uniformly bounded with respect to $n$.
 This fact and the regularity assumptions \eqref{reg} imply
$$\lim_{n\to \infty}\int_0^T\int_{\mathbb{T}^d}\int_{\mathbb{R}^d}\nabla_p \zeta_{\epsilon}(p)\cdot (\nabla_p (f_n-f) +p(f_n+\kappa f^2_n)-p(f+\kappa f^2))\ln\left({\frac{f_n}{\beta_{\infty}M(1+\kappa f_n)}}\right)
  dpdxdt=0.$$ 
  Since $\nabla_p \zeta_{\epsilon}(p)\ln\left({\frac{f_n}{\beta_{\infty}M(1+\kappa f_n)}}\right)$ is uniformly bounded with respect to $n$ (and this uniform bound is integrable on $(0,T)\times\mathbb{T}^d\times\text{supp}(\zeta_\epsilon)$) and converges pointwise a.e.\,to $\nabla_p \zeta_{\epsilon}(p)\ln\left({\frac{f}{\beta_{\infty}M(1+\kappa f)}}\right)$ as $n\to\infty$ (see \cite[Theorem 8.15]{Foll}), the Lebesgue dominated convergence theorem provides 
  \begin{align*}
  &\lim_{n\to \infty} \int_0^T\int_{\mathbb{T}^d}\int_{\mathbb{R}^d}\nabla_p \zeta_{\epsilon}(p)\cdot (\nabla_p f +p(f+\kappa f^2) )\ln\left({\frac{f_n}{\beta_{\infty}M(1+\kappa f_n)}}\right)
  dpdxdt\\
  &=\int_0^T\int_{\mathbb{T}^d}\int_{\mathbb{R}^d}\nabla_p \zeta_{\epsilon}(p)\cdot (\nabla_p f +p(f+\kappa f^2) )\ln\left({\frac{f}{\beta_{\infty}M(1+\kappa f)}}\right)
  dpdxdt.
  \end{align*} 
Hence, this last term is also the limit of the second term on the r.h.s.\ of \eqref{H_e}.
 Since $f_n(t)\to f(t)$ in $L^2(\mathbb{T}^d\times \mathbb{R}^d)$ as $n\to \infty$ for any fixed $t\geq 0$ and $\frac{\zeta_{\epsilon}}{M}$ is bounded, we have $$\lim_{n\to \infty} \int_{\mathbb{T}^d\times \mathbb{R}^d}\zeta_{\epsilon}\frac{(f_n(t)-f(t))^2}{M}dpdx= 0.$$  We can write   $$\mathrm{H}_{\epsilon}[f_n(t)|f_{\infty}]-\mathrm{H}_{\epsilon}[f(t)|f_{\infty}]=\mathrm{H}_{\epsilon}[f_n(t)|f(t)]+\int_{\mathbb{T}^d\times \mathbb{R}^d} \zeta_{\epsilon}(f_n(t)-f(t))\ln{\frac{f(t)(1+\kappa f_{\infty})}{f_{\infty}(1+\kappa f(t))}}dpdx.$$ 
{As mentioned before,} $\mathrm{H}_{\epsilon}[f_n(t)|f(t)]$ is equivalent to $ \int_{\mathbb{T}^d\times \mathbb{R}^d}\zeta_{\epsilon}\frac{(f_n(t)-f(t))^2}{M}dpdx $ and so $\lim_{n\to \infty}\mathrm{H}_{\epsilon}[f_n(t)|f(t)]=0.$ As $f_n(t)\to f(t)$ in $L^1(\mathbb{T}^d\times \mathbb{R}^d)$ as $n\to \infty$ for any fixed $t\geq 0$ and $\zeta_{\epsilon} \ln{\frac{f(t)(1+\kappa f_{\infty})}{f_{\infty}(1+\kappa f(t))}}$ is bounded (see \eqref{bound}), we have $$\lim_{n\to \infty}\int_{\mathbb{T}^d\times \mathbb{R}^d} \zeta_{\epsilon}(f_n(t)-f(t))\ln{\frac{f(t)(1+\kappa f_{\infty})}{f_{\infty}(1+\kappa f(t))}}dpdx=0.$$  The above limits show  $$\lim_{n\to \infty}\mathrm{H}_{\epsilon}[f_n(t)|f_{\infty}]=\mathrm{H}_{\epsilon}[f(t)|f_{\infty}]$$ for any fixed $t\geq 0.$ In particular, we can  pass to the limit in the left hand side of \eqref{H_e}:
 $$\lim_{n\to \infty}\left(\mathrm{H}_{\epsilon}[f_n(T)|f_{\infty}]-\mathrm{H}_{\epsilon}[f_n(0)|f_{\infty}]\right)=\mathrm{H}_{\epsilon}[f(T)|f_{\infty}]-\mathrm{H}_{\epsilon}[f_0|f_{\infty}].$$
 Next we rewrite the {first term on the r.h.s.\ of} \eqref{H_e} as 
 \begin{align*}&\int_0^T\int_{\mathbb{T}^d}\int_{\mathbb{R}^d}\zeta_{\epsilon}(p)\frac{|\nabla_p f_n +p(f_n+\kappa f^2_n)|^2}{f_n(1+\kappa f_n)}dpdxdt\\
 =&{\int_0^T\int_{\mathbb{T}^d}\int_{\mathbb{R}^d}} \zeta_{\epsilon}(p)\frac{|\nabla_p f_n +p(f_n+\kappa f^2_n)|^2-|\nabla_p f +p(f+\kappa f^2)|^2}{f_n(1+\kappa f_n)}dpdxdt\\
& +{\int_0^T\int_{\mathbb{T}^d}\int_{\mathbb{R}^d}} \zeta_{\epsilon}(p)\frac{|\nabla_p f +p(f+\kappa f^2)|^2}{f_n(1+\kappa f_n)}dpdxdt.
 \end{align*}
 Since 
 $\frac{\zeta_{\epsilon}(p)}{f_{n}(1+\kappa f_n)}$ is bounded uniformly with respect to $n$, the regularity assumptions \eqref{reg} provide 
$$\lim_{n\to \infty} {\int_0^T\int_{\mathbb{T}^d}\int_{\mathbb{R}^d}}\zeta_{\epsilon}(p)\frac{|\nabla_p f_n +p(f_n+\kappa f^2_n)|^2-|\nabla_p f +p(f+\kappa f^2)|^2}{f_n(1+\kappa f_n)}dpdxdt=0.$$ 
The sequence $\frac{\zeta_{\epsilon}(p)}{f_{n}(1+\kappa f_n)}$ is bounded {uniformly with respect to $n$} and converges pointwise to  {$\frac{\zeta_{\epsilon}(p)}{f(1+\kappa f)}$} as $n\to \infty$. The Lebesgue dominated convergence theorem shows 
$$\lim_{n\to \infty} {\int_0^T\int_{\mathbb{T}^d}\int_{\mathbb{R}^d}} \zeta_{\epsilon}(p)\frac{|\nabla_p f +p(f+\kappa f^2)|^2}{f_n(1+\kappa f_n)}dpdxdt= {\int_0^T\int_{\mathbb{T}^d}\int_{\mathbb{R}^d}} \zeta_{\epsilon}(p)\frac{|\nabla_p f +p(f+\kappa f^2)|^2}{f(1+\kappa f)}dpdxdt,$$
{which is the limit of the first term on the r.h.s.\ of \eqref{H_e}}.
 These limits allow us to pass to the limit in \eqref{H_e} as $n\to \infty$: 
 \begin{align}\label{lim0}
\mathrm{H}_{\epsilon}[f(T)|f_{\infty}]&-\mathrm{H}_{\epsilon}[f_0|f_{\infty}]
= -\int_0^T\int_{\mathbb{T}^d}\int_{\mathbb{R}^d}\zeta_{\epsilon}(p)\frac{|\nabla_p f +p(f+\kappa f^2)|^2}{f+\kappa f^2}dpdxdt\nonumber\\
-&\int_0^T\int_{\mathbb{T}^d}\int_{\mathbb{R}^d}\nabla_p \zeta_{\epsilon}(p)\cdot (\nabla_p f +p(f+\kappa f^2) )\ln\left({\frac{f}{\beta_{\infty}M(1+\kappa f)}}\right)
  dpdxdt.
\end{align} 

\textbf{Step 2, $\varepsilon\to 0$:} Next we let $\epsilon\to 0.$ First,  \eqref{gr.z} shows that 
\begin{equation}\label{Step2-term}
\nabla_p \zeta_{\epsilon}(p)\cdot (\nabla_p f +p(f+\kappa f^2) )\ln\left({\frac{f}{\beta_{\infty}M(1+\kappa f)}}\right)
\end{equation}
converges pointwise a.e.\,to zero as $\epsilon\to 0.$ Moreover, $|\nabla_p \zeta_{\epsilon}(p)|$ and   $\nabla_p \zeta_{\epsilon}(p)\cdot p$  are bounded for $p \in \mathbb{R}^d,$ uniformly for $\epsilon\in (0,1].$ 
Also, the logarithmic factor in \eqref{Step2-term} is bounded because of the bounds on $f$ in \eqref{bound}, and the regularity assumed in this lemma for $f$ implies integrability of $\nabla_p f$, $f$, and $f^2$.
Hence, the Lebesgue dominated convergence theorem  yields
\begin{equation}\label{lim1}\lim_{\epsilon\to 0}\int_0^T\int_{\mathbb{T}^d}\int_{\mathbb{R}^d}\nabla_p \zeta_{\epsilon}(p)\cdot (\nabla_p f +p(f+\kappa f^2) )\ln\left({\frac{f}{\beta_{\infty}M(1+\kappa f)}}\right) dpdxdt=0.
\end{equation}
The function $\zeta_{\epsilon}(p)\left(f \ln{\frac{f}{f_{\infty}}}-\kappa(1+\kappa f)\ln{\frac{1+\kappa f}{1+\kappa f_{\infty}}}\right)$ is positive (see the proof Lemma \ref{C<H<C}), monotone increasing with respect to $\epsilon,$ and converges to $f \ln{\frac{f}{f_{\infty}}}-\kappa(1+\kappa f)\ln{\frac{1+\kappa f}{1+\kappa f_{\infty}}}$ pointwise as $\epsilon\to 0.$ Hence,  the monotone convergence  theorem provides 
\begin{equation}\label{lim2}\lim_{\epsilon\to 0}(\mathrm{H}_{\epsilon}[f(T)|f_{\infty}]-\mathrm{H}_{\epsilon}[f_0|f_{\infty}])=\mathrm{H}[f(T)|f_{\infty}]-\mathrm{H}[f_0|f_{\infty}].
\end{equation}
Similar arguments provide
\begin{equation}\label{lim3}\lim_{\epsilon\to 0}\int_0^T\int_{\mathbb{T}^d}\int_{\mathbb{R}^d}\zeta_{\epsilon}(p)\frac{|\nabla_p f +p(f+\kappa f^2)|^2}{f+\kappa f^2}dpdxdt=\int_0^T\int_{\mathbb{T}^d}\int_{\mathbb{R}^d}\frac{|\nabla_p f +p(f+\kappa f^2)|^2}{f+\kappa f^2}dpdxdt,
\end{equation}
and the finiteness of the right hand side follows from the $\epsilon-$uniform boundedness of the left hand side.
 The equations \eqref{lim0}, \eqref{lim1}, \eqref{lim2}, and \eqref{lim3} show
\begin{equation}\label{lim4}\mathrm{H}[f(T)|f_{\infty}]-\mathrm{H}[f_0|f_{\infty}]=-\int_0^T\int_{\mathbb{T}^d}\int_{\mathbb{R}^d}\frac{|\nabla_p f +p(f+\kappa f^2)|^2}{f+\kappa f^2}dpdxdt. 
\end{equation}
The generalized log-Sobolev inequality (see \cite[Page 152]{1D} and \cite[Pages 2228-2229]{FDFP}) provides that there is a constant $C>0$ such that
 $$\mathrm{H}[f|\Pi [f]]\leq \frac{1}{C}\int_{\mathbb{T}^d}\int_{\mathbb{R}^d}\frac{|\nabla_p f +p(f+\kappa f^2)|^2}{f+\kappa f^2}dpdx.$$  
By Lemma \ref{C<H<C}, $\mathrm{H}[f|\Pi [f]]$ is equivalent to $||f-\Pi[ f]||^2$. Hence, \eqref{lim4} and the above inequality yield 
\begin{equation*}\mathrm{H}[f(T)|f_{\infty}]-\mathrm{H}[f_0|f_{\infty}]\leq-{C'} \int_0^T\int_{\mathbb{T}^d}\int_{\mathbb{R}^d}\frac{(f-\Pi [f])^2}{M}dpdxdt 
\end{equation*}
for some ${C'}>0,$ which is the claimed estimate.
\end{proof}
%%%%%%%%%%%%%%%%%%%%%%%%%%

\subsection{Proof of Theorem \ref{main th}}
%%%%%%%%%%%%%%%%
%%%%%%%%%%%%%%%%
%%%%%%%%
We denote the macroscopic flux by
$$j(t,x)\colonequals \int_{\mathbb{R}^d}pfdp$$
and recall
$$\rho(t,x)\colonequals \int_{\mathbb{R}^d}fdp.$$
Let $\phi=\phi(t,x)$ be the solution of the Poisson equation
\begin{equation}\label{Pois}
-\Delta_x \phi=\rho-\rho_{\infty},\, \, \, \int_{\mathbb{T}^d} \phi \,dx=0, \, \, \, x\in \mathbb{T}^d, \, \, t\geq 0.
\end{equation}
 Using these functions we define a modified entropy functional
\begin{equation*}
\mathrm{E}[f|f_{\infty}]\colonequals \mathrm{H}[f|f_{\infty}]+\delta \int_{\mathbb{T}^d} \nabla_x \phi\cdot j dx,
\end{equation*}
 where $\delta>0$ will be chosen in the proof of Theorem \ref{main th}. We now show that also $\mathrm{E}[f|f_{\infty}]$ is equivalent to $||f-f_{\infty}||^2.$ 
\begin{lemma}\label{l: E eq}   Let $f$ satisfy \eqref{bound}.
Then, if $\delta>0$ is small enough, there are constants $C_6>0$ and $C_7>0$ (depending only on $\beta_-$ and $\beta_+$) such that \begin{equation}\label{E equiv}C_6||f-f_{\infty}||^2\leq \mathrm{E}[f|f_{\infty}]\leq C_7 ||f-f_{\infty}||^2.
\end{equation}
\end{lemma}
\begin{proof} We estimate  
$$ \left|\int_{\mathbb{T}^d} \nabla_x \phi\cdot j dx\right|\leq ||\nabla_x \phi||_{L^2(\mathbb{T}^d)}||j||_{L^2(\mathbb{T}^d)}.$$
Because of elliptic regularity, $||\nabla_x \phi||_{L^2(\mathbb{T}^d)}$ is bounded (up to a constant) by $||\rho-\rho_{\infty}||_{L^2(\mathbb{T}^d)}.$ Moreover, the H\"older inequality provides $||\rho-\rho_{\infty}||_{L^2(\mathbb{T}^d)}\leq ||f-f_{\infty}||.$  
We use $\displaystyle \int_{\mathbb{R}^d }pf_{\infty}dp=0$ and the H\"older inequality to estimate 
\begin{align*}
||j(t)||_{L^2(\mathbb{T}^d)}&=\sqrt{\int_{\mathbb{T}^d}\left|\int_{\mathbb{R}^d}p(f-f_{\infty})dp \right|^2dx}\\
&\leq \sqrt{\int_{\mathbb{T}^d}\left(\int_{\mathbb{R}^d}|p|^2M(p)dp \right)\left(\int_{\mathbb{R}^d}\frac{(f-f_{\infty})^2}{M}dp\right) dx}\\
&\leq \sqrt{d}||f-f_{\infty}||. 
\end{align*}
These estimates show that there is a positive constant $C$ such that
$$ \left|\int_{\mathbb{T}^d} \nabla_x \phi\cdot j dx\right|\leq C||f-f_{\infty}||^2.$$
This estimate and \eqref{C_1<C_2} show
$$(C_1-\delta C)||f-f_{\infty}||^2\leq \mathrm{E}[f|f_{\infty}]\leq (C_2+\delta C)||f-f_{\infty}||^2.$$
This shows that, if $\delta$ is small enough, $\mathrm{E}[f|f_{\infty}]$ is equivalent to $||f-f_{\infty}||^2.$
\end{proof}
%%%%%%%%%%%%%%%%%%%%%%%%%%

%%%%%%%%%%%%%%%%%
%%%%%%%%%%%%%%%%%%%%%%%%%%%%%%%%%%%%%%%%%%%%%%%
We will need the following general functional analysis fact whose  proof can be found, for instance, in \cite[Lemma 1.2, page 260]{Temam}.
  \begin{lemma}[\bf{Lions-Magenes lemma}]\label{le:partial_t rho}
  Let $\mathcal{V},\mathcal{H},\mathcal{V}'$  be three Hilbert spaces, each space included in the following one with dense and bounded embedding, $\mathcal{V}'$ being the dual of $\mathcal{V}$.   Then, for $T>0$, 
  the following inclusion 
  $$
  L^2([0,T],\mathcal{V})\cap H^1([0,T],\mathcal{V}')\subset C([0,T],\mathcal{H})
  $$
  holds true. Moreover, for any $g \in L^2([0,T],\mathcal{V})\cap H^1([0,T],\mathcal{V}')$ there holds 
  $$
  t\rightarrow ||g(t)||^2_{\mathcal{H}} \in W^{1,1}(0,T)
  $$
  and
  \begin{equation*}
  \frac{d}{dt} ||g(t)||^2_{\mathcal{H}}=2\langle g'(t),g(t) \rangle_{\mathcal{V}',\mathcal{V}}  \,\,\,\, a.e. \text{    on   }   (0,T),
  \end{equation*}
  where $||\cdot||_\mathcal{H}$ denotes the norm in $\mathcal{H},$ and $\langle \cdot, \cdot \rangle_{\mathcal{V}',\mathcal{V}}$ denotes the action of $\mathcal{V}'$ on $\mathcal{V}.$
  \end{lemma}
   One of the application of this lemma is that if $g_1\in L^2([0,T],\mathcal{H})\cap H^1([0,T],\mathcal{V}')$ and $g_2 \in L^2([0,T],\mathcal{V})\cap H^1([0,T], \mathcal{H}),$ then 
  $$\frac{d}{dt}\langle g_1(t),g_2(t)\rangle_{\mathcal{H}}=\langle g'_1(t),g_2(t) \rangle_{\mathcal{V}',\mathcal{V}} +\langle g_1(t),g'_2(t) \rangle_{\mathcal{H}}, $$
  where $\langle \cdot , \cdot \rangle_{\mathcal{H}}$ denotes the scalar product in $\mathcal{H}.$ For more details we refer to \cite{LM, Temam}. 
  In the following we use this equation for  the case $\mathcal{V}\colonequals H^1(\mathbb{T}^d)$ and $\mathcal{H}\colonequals L^2(\mathbb{T}^d).$
%\section{Proof of Theorem \ref{main th}}

\begin{proof}[\textbf{Proof of Theorem \ref{main th}}]
The inequalities in \eqref{-f+} have been proven in  Corollary \ref{Cor:bound}. Hence, it remains to prove  \eqref{lamda}.

We consider the modified entropy functional \begin{equation*}
\mathrm{E}[f(t)|f_{\infty}]=\mathrm{H}[f(t)|f_{\infty}]+\delta \int_{\mathbb{T}^d} \nabla_x \phi\cdot j dx.
\end{equation*}
We have estimated $\mathrm{H}[f(t)|f_{\infty}]$ in \eqref{dt H<-C_5}. We want to get a similar estimate for  $\int_{\mathbb{T}^d} \nabla_x \phi\cdot j dx.$  To do that we consider its time derivative $\frac{d}{dt}\int_{\mathbb{T}^d} \nabla_x \phi\cdot j dx.$ Multiplying \eqref{Eq} by $p,$  then integrating formally with respect to $p,$ we obtain
\begin{equation}\label{dt j -1}
\partial_t j=-\mathrm{div}_x\int_{\mathbb{R}^d}p\otimes p(f+\kappa f^2)dp-\int_{\mathbb{R}^d}p(f+\kappa f^2)dp,
\end{equation}
where $\otimes$ denotes the Kronecker product.
\eqref{bound} implies that the right hand side of this equation is in $H^{-1}(\mathbb{T}^d)$ for all $t\geq 0.$ 
%$$\int_{\mathbb{T}^d} \partial_t\nabla_x \phi\cdot j dx+ \int_{\mathbb{T}^d} \nabla_x \phi\cdot \partial_t j dx.$$
By formally integrating \eqref{Eq} with respect to $p,$  we obtain
\begin{equation}\label{dt rho}
\partial_t \rho=-\mathrm{div}_x\left(\int_{\mathbb{R}^d} p(f+\kappa f^2)dp\right).
\end{equation}
\eqref{bound} implies that $\partial_t \rho\in H^{-1}(\mathbb{T}^d)$ for all $t\geq 0.$
  Also, by differentiating \eqref{Pois} with respect to $t>0$, we get 
  $\partial_{t}(\Delta_x \phi)=-\partial_t \rho\in H^{-1}(\mathbb{T}^d) $, and so  $\partial_t   \phi\in H^1(\mathbb{T}^d)$ for all $t\geq 0.$   Therefore, we have 
\begin{equation}\label{1/4} \frac{d}{dt}\int_{\mathbb{T}^d} j\cdot \nabla_x \phi dx=\langle \partial_t j,\nabla_x \phi\rangle_{H^{-1}, H^1}+\int_{\mathbb{T}^d} \partial_{t} \nabla_x \phi \cdot j dx.
\end{equation}
%Since $f$ solves \eqref{Eq}, the time derivative of this functional equals
%\begin{align}\label{dt E}
%\frac{d}{dt}\mathrm{E}[f(t)|f_{\infty}]=\frac{d}{dt}\mathrm{H}[f(t)|f_{\infty}]+\delta \int_{\mathbb{T}^d} \partial_t\nabla_x \phi\cdot j dx+\delta \int_{\mathbb{T}^d} \nabla_x \phi\cdot \partial_t j dx
%\end{align}
We estimate  this time derivative  in the following two steps.\\
\textbf{Step 1, Estimates of $\displaystyle \int_{\mathbb{T}^d} \partial_t\nabla_x \phi\cdot j dx$ : }
The H\"older inequality implies  
$$\left|\int_{\mathbb{T}^d} \partial_t\nabla_x \phi\cdot j dx\right|\leq ||\partial_t \nabla_x \phi||_{L^2(\mathbb{T}^d)} ||j||_{L^2(\mathbb{T}^d)}. $$
We use \eqref{pP=0} and the H\"older inequality to obtain 
$$||j(t)||^2_{L^2(\mathbb{T}^d)}=\int_{\mathbb{T}^d}\left|\int_{\mathbb{R}^d}p(f-\Pi [f])dp \right|^2dx\leq d\int_{\mathbb{T}^d}\int_{\mathbb{R}^d}\frac{(f-\Pi[ f])^2}{M}dp dx. $$
Using  $\partial_t \Delta_x \phi=-\partial_t \rho\in H^{-1}(\mathbb{T}^d)$ and \eqref{dt rho}, we compute
\begin{align*}
||\partial_t \nabla_x \phi(t)||_{L^2(\mathbb{T}^d)}^2&=-\langle \partial_t \Delta_x \phi, \partial_t  \phi  \rangle_{H^{-1},H^1}\\
&=-\Big \langle \text{div}_x\left(\int_{\mathbb{R}^d} p(f+\kappa f^2)dp\right), \partial_t  \phi\Big\rangle_{H^{-1},H^1} \\
&=\int_{\mathbb{T}^d}\partial_t  \nabla_x \phi \cdot \left(\int_{\mathbb{R}^d} p(f+\kappa f^2)dp\right)dx. 
\end{align*}
The H\"older inequality implies
$$||\partial_t \nabla_x \phi(t)||_{L^2(\mathbb{T}^d)}\leq \sqrt{ \int_{\mathbb{T}^d}\left|\int_{\mathbb{R}^d} p(f+\kappa f^2)dp\right|^2dx}. $$ 
 We use \eqref{vP=0} and the H\"older inequality to obtain 
\begin{align}\label{bef50} \int_{\mathbb{T}^d}\left|\int_{\mathbb{R}^d} p(f+\kappa f^2)dp\right|^2dx &=\int_{\mathbb{T}^d}\left|\int_{\mathbb{R}^d}p (f-\Pi [f])(1+\kappa f+\kappa \Pi [f])dp \right|^2dx \nonumber \\
&\leq d||1+\kappa f+\kappa \Pi [f]||^2_{L^{\infty}(\mathbb{T}^d\times \mathbb{R}^d)}\int_{\mathbb{T}^d}\int_{\mathbb{R}^d}\frac{(f-\Pi [f])^2}{M}dp dx. 
\end{align}
The factor $||1+\kappa f+\kappa \Pi [f]||^2_{L^{\infty}(\mathbb{T}^d\times \mathbb{R}^d)}$ can be bounded by a positive  constant depending on $\beta_-$ and $\beta_+.$ These estimates show that there is a constant $C>0$ such that  
\begin{equation}\label{tphi j}\left|\int_{\mathbb{T}^d} \partial_t\nabla_x \phi\cdot j dx\right|\leq C||f-\Pi [f]||^2.
\end{equation}\\
\textbf{Step 2, Estimates of $\displaystyle \langle \partial_t j, \nabla_x \phi\rangle_{H^{-1}, H^1}$ :}
 We use
\eqref{v_iv_jP} to write \eqref{dt j -1} as   
 \begin{align*}
 \partial_t j=&-\nabla_x \rho- \mathrm{div}_x\int_{\mathbb{R}^d}p \otimes p \{(f-\Pi [f])(1+\kappa f+\kappa \Pi [f])\}dp-\int_{\mathbb{R}^d}p(f+\kappa f^2)dp\\
 =&-\nabla_x (\rho-\rho_{\infty})- \mathrm{div}_x\int_{\mathbb{R}^d}p\otimes p \{(f-\Pi [f])(1+\kappa f+\kappa \Pi [f])\}dp-\int_{\mathbb{R}^d}p(f+\kappa f^2)dp.
 \end{align*}
 We use this equation to compute 
 \begin{align}\label{tj}
 \langle \partial_t j, \nabla_x \phi\rangle_{H^{-1}, H^1}=&-\langle  \nabla_x(\rho-\rho_{\infty}), \nabla_x \phi \rangle_{H^{-1},H^1}\nonumber\\
 &-\Big \langle \mathrm{div}_x\int_{\mathbb{R}^d}p\otimes p \{(f-\Pi [f])(1+\kappa f+\kappa \Pi [f])\}dp, \nabla_x \phi \Big \rangle_{H^{-1},H^1}\nonumber\\
 &-\int_{\mathbb{T}^d}\int_{\mathbb{R}^d}\nabla_x \phi\cdot p(f+\kappa f^2)dpdx.
 \end{align}
 For the first term on the right hand side we  use  \eqref{rr} to get 
 \begin{equation}\label{tj1}
 -\langle  \nabla_x(\rho-\rho_{\infty}), \nabla_x \phi \rangle_{H^{-1},H^1} =-\int_{\mathbb{T}^d}(\rho-\rho_{\infty})^2dx \leq -C_4||\Pi [f]-f_{\infty}||^2.
 \end{equation}
 For the second term in \eqref{tj}, we  use the H\"older inequality
 \begin{align*}
&-\Big \langle \mathrm{div}_x\int_{\mathbb{R}^d}p\otimes p \{(f-\Pi [f])(1+\kappa f+\kappa \Pi [f])\}dp, \nabla_x \phi \Big \rangle_{H^{-1},H^1}\\
 &=\int_{\mathbb{T}^d}\int_{\mathbb{R}^d} p^T\frac{\partial^2 \phi}{\partial x^2} p (f-\Pi [f])(1+\kappa f+\kappa \Pi [f])dpdx\\
 & \leq ||1+\kappa f+\kappa \Pi [f]||_{L^{\infty}(\mathbb{T}^d\times \mathbb{R}^d)}||f-\Pi [f]||\sqrt{\int_{\mathbb{T}^d}\int_{\mathbb{R}^d} |p|^2\left|\frac{\partial^2 \phi}{\partial x^2}\right|^2 Mdpdx}\\
 &=\sqrt{d}||1+\kappa f+\kappa \Pi [f]||_{L^{\infty}(\mathbb{T}^d\times \mathbb{R}^d)}||f-\Pi [f]||\left|\left|\frac{\partial^2 \phi}{\partial x^2}\right|\right|_{L^2(\mathbb{T}^d)}.
 \end{align*}
The norm $\left|\left|\frac{\partial^2 \phi}{\partial x^2}\right|\right|_{L^2(\mathbb{T}^d)}$ is bounded (up to a constant) by $||\rho-\rho_{\infty}||_{L^2(\mathbb{R}^d)}$ because of elliptic regularity. Because of  \eqref{rr}    the norm $||\rho-\rho_{\infty}||_{L^2(\mathbb{T}^d)}$ is bounded (up to a constant) by $||\Pi [f]-f_{\infty}||.$ Moreover, $||1+\kappa f+\kappa \Pi [f]||_{L^{\infty}(\mathbb{T}^d\times \mathbb{R}^d)}$ is bounded by a positive  constant depending $\beta_-$ and $\beta_+.$ Hence, there is a constant $C'>0$ such that 
 \begin{equation}\label{tj2}
  -\Big \langle \mathrm{div}_x\int_{\mathbb{R}^d}p\otimes p \{(f-\Pi [f])(1+\kappa f+\kappa \Pi [f])\}dp, \nabla_x \phi \Big \rangle_{H^{-1},H^1}
  \leq C'  || f-\Pi [f]||\,||\Pi [f]-f_{\infty}||.
 \end{equation}
 We estimate the last term in \eqref{tj} by using the H\"older inequality and \eqref{bef50}:
 \begin{align*}
 -\int_{\mathbb{T}^d}\int_{\mathbb{R}^d}\nabla_x \phi\cdot p(f+\kappa f^2)dpdx & \leq ||\nabla_x \phi||_{L^2(\mathbb{T}^d)} \sqrt{\int_{\mathbb{T}^d}\left|\int_{\mathbb{R}^d}p(f+\kappa f^2)dp\right|^2dx}\\
 &\leq \sqrt{d}||1+\kappa f+\kappa \Pi [f]||_{L^{\infty}(\mathbb{T}^d\times \mathbb{R}^d)}||\nabla_x \phi||_{L^2(\mathbb{T}^d)}\,||f-\Pi [f]||.
 \end{align*}
 Similarly, $||\nabla_x \phi||_{L^2(\mathbb{T}^d)}$  is bounded (up to a constant) by $||\rho-\rho_{\infty}||_{L^2(\mathbb{R}^d)}$ because of elliptic regularity. As before,   $||\rho-\rho_{\infty}||_{L^2(\mathbb{T}^d)}$ is bounded (up to a constant) by $||\Pi [f]-f_{\infty}||$ and  $||1+\kappa f+\kappa \Pi [f]||^2_{L^{\infty}(\mathbb{T}^d\times \mathbb{R}^d)}$ is  bounded by a positive  constant.
 Hence, there is a constant $\tilde{C}>0$ such that 
 \begin{equation}\label{tj3}
 -\int_{\mathbb{T}^d}\int_{\mathbb{R}^d}\nabla_x \phi\cdot p(f+\kappa f^2)dpdx  \leq \tilde{C}||f- \Pi [f]||\,||\Pi [f]-f_{\infty}||.
 \end{equation}
The equations \eqref{tj}, \eqref{tj1}, \eqref{tj2}, and \eqref{tj3} imply 
 \begin{equation*}\langle \partial_t j, \nabla_x \phi\rangle_{H^{-1}, H^1}\leq -C_4||\Pi [f]-f_{\infty}||^2+(C' +\tilde{C})|| f-\Pi [f]||\,||\Pi [f]-f_{\infty}||.
 \end{equation*}
Altogether, Equation \eqref{1/4}, the above estimate,  and \eqref{tphi j} yield 
  \begin{multline}\label{phitj}\int_{\mathbb{T}^d} \nabla_x \phi(T)\cdot  j(T) dx\leq \int_{\mathbb{T}^d} \nabla_x \phi(0)\cdot  j(0) dx 
  +\int_0^T\Big\{-C_4||\Pi [f]-f_{\infty}||^2+C || f-\Pi [f]||^2\\+(C'+\tilde{C}) || f-\Pi [f]||\,||\Pi [f]-f_{\infty}||\Big\}dt
 \end{multline}
 for all $T>0.$\\
 \textbf{Step 3, Gr\"onwall's inequality:} We now summarize all estimates.  \eqref{dt H<-C_5} and \eqref{phitj} yield
 \begin{align*}\mathrm{E}[f(T)|f_{\infty}]- \mathrm{E}[f(0)|f_{\infty}]\leq &-\int_0^T\Big\{(C_5-\delta C)||f-\Pi [f]||^2 \\&+\delta C_4||\Pi [f]-f_{\infty}||^2
 -\delta (C'+\tilde{C}) || f-\Pi [f]||\,||\Pi [f]-f_{\infty}||\Big\}dt.
 \end{align*}
 One can check that if $\delta>0$ is small enough, there is a constant $C''>0$ such that 
 $$\mathrm{E}[f(T)|f_{\infty}]-\mathrm{E}[f(0)|f_{\infty}]\leq -2C''\int_0^T(||f-\Pi [f]||^2+||\Pi [f]-f_{\infty}||^2)dt\leq  -C''\int_0^T||f-f_{\infty}||^2dt.$$
 We choose $\delta>0$ so that this inequality and \eqref{E equiv} hold.
 The first inequality in \eqref{E equiv} implies
 $$\mathrm{E}[f(T)|f_{\infty}]- \mathrm{E}[f(0)|f_{\infty}]\leq -\frac{C''}{C_6}\int^T_0\mathrm{E}[f(t)|f_{\infty}]dt.$$
 Since this inequality holds for all $T>0$, Gr\"onwall's inequality implies 
 $$\mathrm{E}[f(t)|f_{\infty}]\leq e^{-\frac{C''}{C_6}t}\mathrm{E}[f_0|f_{\infty}], \, \, \, \forall\, t\geq 0.$$
 This inequality and \eqref{E equiv} provide \eqref{lamda} with $c\colonequals \sqrt{\frac{C_7}{C_6}}$ and $2\lambda\colonequals \frac{C''}{C_6}.$
\end{proof}
 
%%%%%%%%%%%%%%%%%%%%%%%%%%%%%%%%%%%%%%%%%%%%%%%%%%%%%%%%%%%%%%%%%%%%%%%%%%%%%%%%%%%%

\begin{Remark} 
    The regularity assumption \eqref{reg} on the weak solution is needed  to prove the $L^1$-contraction in Lemma \ref{max.pri} and to estimate the time derivative of the relative entropy in Lemma \ref{l:Ent}. More precisely, \eqref{reg} is needed to prove   the limit \eqref{com_p} in Lemma \ref{l:R_n}. Although the $L^1$-contraction and the estimate on the time derivative of the relative entropy do not include the first order derivatives of the solution, their proofs require passing to limits in some terms using  the limit \eqref{com_p} in Lemma \ref{l:R_n}. %Hence, this regularity assumption is required to prove them. 
 However, the mass conservation in Lemma \ref{l:mass} is proven without this regularity assumption.

  The proof of Theorem \ref{main th} shows that the  regularity assumption \eqref{reg} on the weak solution {could} be replaced by the condition that this solution satisfies the bound  \eqref{bound} in Corollary 3.6 and the estimate \eqref{dt H<-C_5} in Lemma 4.3. 
\end{Remark}

%%%%%%%%%%%%%%%%%%%%%%%%%%%%%%%%%%%%%%%%%%%%%%%%%%%%%%%%%%
%%%%%%%%%%%%%%%%%%%%%%
\section{Conclusion and Outlook}
%%%%%%%%%%%%%%%%%%%%%%%%%%%%%%%%%%%%%%%%%%%%%%

In this paper, we investigated the long-time behavior of weak solutions to the spatially inhomogeneous Fokker-Planck equation \eqref{Eq} for bosons and fermions. Under suitable regularity assumptions on the weak solutions, we established convergence to the global equilibrium without imposing any close-to-equilibrium assumption on the initial data, extending the existing hypocoercivity results \cite{NS, Chi} that needed such assumption.

Our main results can be summarized as follows:
\begin{enumerate}
\item[(a)] \textbf{Properties of weak solutions.}
We established several fundamental properties of weak solutions: conservation of mass, $L^1$-contraction, a maximum principle, and decay of the entropy functional.
\item[(b)] \textbf{Exponential convergence to equilibrium.}
We proved exponential convergence to the global equilibrium in a suitably weighted $L^2$ space. The proof is based on the construction of a Lyapunov functional that combines the relative entropy with ideas from the $L^2$-hypocoercivity method of \cite{DMS}. In particular, the functional is designed to capture the interplay between the dissipative momentum diffusion and the transport dynamics, despite the nonlinear structure of both parts of \eqref{Eq}.
\end{enumerate}
Several interesting questions remain open. First, in the present work we assume the existence of weak solutions satisfying certain additional regularity properties. The existence, uniqueness, and regularity theory for such solutions is largely open. Establishing a global well-posedness theory for \eqref{Eq}, in particular for general admissible initial data, would therefore be a natural next step. 
Second, our convergence result is formulated in a weighted
 $L^2$ space and relies on the $L^2$-hypocoercive method. It would be interesting to develop a corresponding $H^1$-hypocoercivity theory and to obtain exponential convergence to equilibrium in a weighted $H^1$ space without the close-to-equilibrium assumption.  Such a result would provide stronger information on the spatial and momentum regularity of solutions and would further clarify the interplay between {nonlinear} dissipation and transport.
Another natural direction concerns the hypoelliptic properties of \eqref{Eq}. Classical kinetic Fokker--Planck equations exhibit hypoelliptic regularization, with quantitative smoothing estimates; see, for example, \cite{AT1, AT2, V}. It remains an interesting question to determine to what extent this property persists for the nonlinear semiclassical equation \eqref{Eq}. In particular, establishing optimal quantitative regularization estimates  would be of considerable interest.

More broadly, semiclassical kinetic equations are mathematically more intricate than their classical counterparts because  quantum statistical effects introduce additional nonlinearities and lead to the Fermi-Dirac and Bose-Einstein equilibrium states,
 which prevents a direct application of the standard hypocoercivity methods. Our results demonstrate that, despite these additional difficulties, the convergence to equilibrium can be proven in this setting after a suitable modification of the $L^2$-hypocoercivity method.
We expect that the techniques developed in this paper can be useful for studying the long-time behavior of other semiclassical kinetic equations, for example, the semi-classical Landau and Boltzmann equations.

%%%%%%%%%%%%%%%%%%%
\bigskip

\textbf{Acknowledgement.}
Anton Arnold was funded in part by the Austrian Science Fund (FWF)
project  \href{https://doi.org/10.55776/F65}{10.55776/F65}. 
 Marlies Pirner and Gayrat Toshpulatov  were funded  by the Deutsche Forschungsgemeinschaft (DFG, German Research Foundation) under Germany's Excellence Strategy EXC 2044/2--390685587, Mathematics Münster: Dynamics--Geometry--Structure. In addition, Marlies Pirner was funded by the German Science Foundation DFG (grant no. PI 1501/2-1).
 For open-access purposes, the authors have applied a CC BY public copyright license to any author-accepted manuscript version arising from this submission. The authors are grateful to the detailed suggestions by the anonymous reviewers that greatly helped to improve this manuscript.
 \bigskip
 
\textbf{Data Availability.} There is no data associated with the paper.
 \bigskip

\textbf{Conflict of interest.} The authors have no conflict of interest to declare.

%%%%%%%%%%APA%%%%%%%%%%%%%%%%%%%%%%%%%%%%%

%%%%%%%%%%%%%%%%%%%%%%%%%%%%%%%%%%%%%%%%%%%%%%%%%%%%%%%%%%%%%%%%%%%%%%%%%%%%%%%%%%%%%%%%%%%%%%%%%%%%%%%%%%%%%%%%%%%%%%%%%%%%%%%%%%%%%%%%%%%%%%%%%%%%%%%%%%%%%%%%%%%%%%%%%%%%%%%%%%%%%%%%%%%%%%%%%%%%%%%%%%%%%%%%%%%%%%%%%%%%%%%%%%%%%%%%%%%%%%%%%%%%%%%%%%%%%%%%%%%%%%%%%%%%%%%%%%%%%

{}

\end{document}